\documentclass{article}

\usepackage[preprint]{neurips_2026}

\usepackage[utf8]{inputenc}
\usepackage[T1]{fontenc}
\usepackage{hyperref}
\usepackage{url}
\usepackage{booktabs}
\usepackage{amsfonts}
\usepackage{nicefrac}
\usepackage{microtype}
\usepackage{xcolor}

\usepackage[hyperpageref]{backref}
\hypersetup{
    colorlinks=true,
    citecolor=darkscarlet,
    linkcolor=darkpowderblue,
    filecolor=magenta,
    urlcolor=yaleblue,
    menucolor=gray
}
\renewcommand*{\backref}[1]{}%
\renewcommand*{\backrefalt}[4]{%
   \ifcase #1
     \footnotesize{(Not cited.)}%
   \or
     \footnotesize{(Cited on page~#2)}%
   \else
     \footnotesize{(Cited on page~#2)}%
\fi }

\PassOptionsToPackage{capitalise,noabbrev}{cleveref}

\usepackage{multirow}
\usepackage[utf8]{inputenc}
\usepackage[T1]{fontenc}
\usepackage{calligra}
\usepackage{layout}

\usepackage{amsmath}
\usepackage{amssymb}
\usepackage{amsfonts}
\usepackage{amsthm}

\usepackage{mathtools}
\usepackage{thmtools} 
\usepackage{thm-restate}

\usepackage{latexsym}

\usepackage{xcolor}
\usepackage{graphicx}

\usepackage{algorithmic}
\usepackage{algorithm}

\usepackage{hyperref}
\usepackage{url}
\usepackage{booktabs}
\usepackage{microtype}

\usepackage{cleveref}

\usepackage{bbm}
\usepackage{bm}

\usepackage{nicefrac}
\usepackage{adjustbox}
\usepackage{threeparttable}

\usepackage{import, ifthen}

\usepackage{tcolorbox}
\usepackage{pifont}
\definecolor{mydarkred}{RGB}{192,25,25}
\definecolor{mydarkgreen}{RGB}{25,192,25}
\definecolor{mydarkblue}{RGB}{25,25,192}

\definecolor{RedOrange}{cmyk}{0,0.77,0.87,0}
\definecolor{Orchid}{cmyk}{0.32,0.64,0,0}

\usepackage{xspace}

\newcommand{\algname}[1]{{\color{gray}\small\sf#1}\xspace}

\newcommand{\norm}[1]{{\left\| #1 \right\|}}

\newcommand{\rbr}[1]{\left(#1\right)}
\newcommand{\sbr}[1]{\left[#1\right]}

\newcommand{\anglebr}[1]{\left\langle#1\right\rangle}

\newcommand{\curlybr}[1]{\left\{#1\right\}}

\newcommand{\inner}[2]{\anglebr{#1, #2}}

\newcommand{\cB}{\mathcal{B}}
\newcommand{\cC}{\mathcal{C}}

\newcommand{\cH}{\mathcal{H}}
\newcommand{\cI}{\mathcal{I}}
\newcommand{\cJ}{\mathcal{J}}

\newcommand{\cO}{\mathcal{O}}

\newcommand{\cX}{\mathcal{X}}

\newcommand{\R}{\mathbb{R}}

\newcommand{\del}[1]{}

\newcommand{\st}{\;:\;}
\newcommand{\eqdef}{\coloneqq}

\DeclareMathOperator{\diam}{diam}

\DeclareMathOperator{\dist}{dist}

\DeclareMathOperator{\lmo}{LMO}

\newcommand{\Proj}{{\rm Proj}}

\newcommand{\pr}[1][]{
  \ifthenelse { \equal{#1}{} }
  { \ensuremath{\mathrm{P}} }
  { \ensuremath{\mathrm{P}\left(#1\right)} }
}

\DeclareMathOperator{\interior}{int}
\DeclareMathOperator{\ri}{ri}

\DeclareMathOperator{\cl}{cl}

\newcommand{\argmin}{\mathop{\arg\!\min}}

\usepackage[colorinlistoftodos,bordercolor=orange,backgroundcolor=orange!20,linecolor=orange,textsize=scriptsize]{todonotes}

\newcommand{\ppeter}[1]{}

\newcommand{\abs}[1]{| #1 |}

\newcommand{\hidesolutions}[1]{}

\definecolor{darkscarlet}{rgb}{0.34, 0.01, 0.1}
\definecolor{yaleblue}{rgb}{0.06, 0.3, 0.57}
\definecolor{darkpowderblue}{rgb}{0.0, 0.2, 0.6}
\definecolor{midnightblue}{HTML}{0059b3}
\definecolor{noonblue}{HTML}{e5eef7}
\definecolor{chromered}{HTML}{f14233}
\definecolor{olivedrab}{HTML}{6b8e23}

\setcitestyle{authoryear,round,citesep={;},aysep={,},yysep={;}}

\allowdisplaybreaks

\declaretheorem[name=Theorem]{theorem}
\crefname{theorem}{Theorem}{Theorems}\Crefname{theorem}{Theorem}{Theorems}

\declaretheorem[sibling=theorem,name=Proposition,style=plain]{proposition}
\declaretheorem[sibling=theorem,name=Lemma,style=plain]{lemma}
\declaretheorem[sibling=theorem,name=Corollary,style=plain]{corollary}
\declaretheorem[sibling=theorem,name=Definition,style=definition]{definition}
\declaretheorem[sibling=theorem,name=Assumption,style=definition]{assumption}
\declaretheorem[sibling=theorem,name=Remark,style=remark]{remark}

\crefname{proposition}{Proposition}{Propositions}\Crefname{proposition}{Proposition}{Propositions}
\crefname{lemma}{Lemma}{Lemmas}\Crefname{lemma}{Lemma}{Lemmas}
\crefname{corollary}{Corollary}{Corollaries}\Crefname{corollary}{Corollary}{Corollaries}
\crefname{definition}{Definition}{Definitions}\Crefname{definition}{Definition}{Definitions}
\crefname{assumption}{Assumption}{Assumptions}\Crefname{assumption}{Assumption}{Assumptions}
\crefname{remark}{Remark}{Remarks}\Crefname{remark}{Remark}{Remarks}

\newcommand{\LLMO}{\algname{Local LMO}}

\renewcommand{\eqref}[1]{\textup{(\ref{#1})}}

\title{Local LMO is Secretly a Projection Method!}

\author{
Peter Richt\'{a}rik \\ KAUST\thanks{King Abdullah University of Science and Technology, Thuwal, Saudi Arabia}
\And
Ammar Mahran \\ KAUST$^*$
}

\begin{document}

\maketitle

\begin{abstract}
The local linear minimization oracle \citep{ferris1996linear,richtarik2026local}, or \algname{Local LMO}, solves constrained convex problems without having to compute a projection:
it minimizes a linear model over the intersection of the feasible set with a ball around the current iterate.
We show that, whenever the ball radius does not exceed the Polyak radius, the \algname{Local LMO} step is the Euclidean projection of the current iterate onto the intersection of the feasible set with a half-space that separates the iterate from the solution set---a projection that is nonetheless computable by a linear oracle alone.
We demonstrate that \algname{Local LMO} belongs to a broader family of projection methods which may be indexed by the depth of the localizing half-space.
For objectives with $\vartheta$-Hölder continuous gradient, every method from this family whose half-space lies sufficiently deep drives the best of its first $K$ iterates to optimality at the universal rate $\cO(K^{-(1+\vartheta)/2})$, matching the non-accelerated universal gradient method of \citet{nesterov2015universal}.
Run at the Polyak radius, \algname{Local LMO} attains the same rate when the constrained optima are also unconstrained ($\norm{\nabla f(x_\star)} = 0$), and the rate $\cO(K^{-1/(2-\vartheta)})$ otherwise.
\end{abstract}

\section{Introduction}
\label{sec:introduction}

We study the constrained convex optimization problem
\begin{equation}
\label{eq:problem}
    \min_{x\in\cX} f(x).
\end{equation}

Problems of the form \eqref{eq:problem} are ubiquitous in optimization \citep{boyd2004convex} and arise throughout machine learning \citep{pokutta2020deep, sangalli2021constrained}, statistics and finance \citep{hathaway1985constrained, hall1999density, behr2013portfolio} and signal processing \citep{zhong2026stability, mattingley2010real}, where constraints encode prior structure, enforce regularization or reflect physical feasibility; in safety-critical deployments, demands for safe, robust and fair models \citep{donini2018empirical, yang2022safety} take the same form \citep{cotter2019optimization, dai2023safe}.
There is a rich optimization literature offering a multitude of algorithms for handling such problems. 
A useful way of mapping out the landscape of these methods is to group them by the oracle used to access the constraint set $\cX$. 
Most approaches fall into one of the following three categories:
\begin{enumerate}
    \item 
    {\bf Separation oracle.}
    Such oracles provide, for a given query point~$x$, either a certificate that $x \in \cX$ or a hyperplane separating $x$ and $\cX$. 
    Classical methods based on this oracle include the ellipsoid method \citep{Grotschel1993Geometric}, cutting-plane schemes \citep{Kelley1960Cutting}, and bundle-type epigraph algorithms \citep{Lemarechal1995Variants}. 
    Membership oracles \citep{Grotschel1993Geometric} are a weaker primitive here, usually converted into separation access.

    \item
    {\bf Projection oracle.} 
    Methods in this class assume access to a projection operator of the form $\Proj_{\cX}^{\phi}(x) \in \argmin_{z \in \cX} D_{\phi}(z, x)$, taken with respect to the Bregman divergence of a strictly convex, differentiable function $\phi$.
    This includes \emph{Projected Gradient Descent} \citep{bertsekas1999nonlinear} when $\phi(x)=\tfrac{1}{2}\norm{x}^2$, mirror descent \citep{Nemirovsky1983problem} (and its special cases such as exponentiated gradient methods \citep{Kivinen1997Exponentiated}), and dual averaging \citep{nesterov2009primal}.

    \item 
    {\bf Linear minimization oracle (LMO).}
    These methods, which include the Frank--Wolfe method \citep{FrankWolfe1956, levitin1966constrained} and its variants, access the constraint set only through linear optimization queries of the form
    \begin{align}\label{eq:lmo}
        \lmo_{\cX}(\nabla f(x)) \in \argmin\limits_{z\in \cX}\inner{\nabla f(x)}{z} ,
    \end{align}
    whenever this minimum is attained.
    Such oracles are attractive in settings where projections may be expensive or intractable.
\end{enumerate}

The local linear minimization oracle \citep{ferris1996linear,richtarik2026local}, or \algname{Local LMO}, falls into the third category and solves constrained convex problems without explicit recourse to a projection oracle.
It minimizes a linear model over the intersection of the feasible set $\cX$ with a ball around the current iterate.
If we denote the current iterate by $x_k$ and the radius of the ball by $t_k > 0$, then the next iterate $x_{k+1}$ computed by \LLMO is
\begin{equation}\label{eq:local_lmo}
    x_{k+1}\ \in\ \argmin_{z\in\cX\cap\cB(x_k,t_k)}\inner{\nabla f(x_k)}{z},
\end{equation}
where $\cB(x,t) \eqdef \curlybr{z\in\R^d\st\norm{z-x}\le t}$ denotes the Euclidean ball of radius $t$ centered at $x$.
The update \eqref{eq:local_lmo} \emph{apparently} avoids projections.

To better understand the workings of \LLMO, we shall work under the following standing assumption.
\begin{assumption}\label{ass:main}
    The set $\cX \subseteq \R^d$ is non-empty, closed and convex, the function $f: \R^d \to \R$ is convex and differentiable, and the solution set $\cX_\star \eqdef \argmin_{z \in \cX} f(z)$ is non-empty.
\end{assumption}
We write $f_\star \eqdef \min_{z \in \cX} f(z)$ for the optimal function value,
and note that the set of optimizers $\cX_\star = \cX \cap \curlybr{z \in \R^d \st f(z) \le f_\star}$ is then closed and convex.

Our main result, \Cref{thm:lmo_is_projection}, identifies the \LLMO step \eqref{eq:local_lmo} with a projection: when the ball radius $t_k$ does not exceed the \emph{Polyak radius},
\begin{equation}
    \label{def:polyak-radius}
    t(x_k) \eqdef \frac{f(x_k)-f_\star}{\norm{\nabla f(x_k)}} ,
\end{equation}
the \LLMO step \eqref{eq:local_lmo} is the Euclidean projection of $x_k$ onto the intersection of the feasible set $\cX$ with a half-space whose outward normal is the gradient $\nabla f(x_k)$ and which separates the current iterate $x_k$ from the solution set $\cX_\star$.
A single linear minimization over $\cX \cap \cB(x_k,t_k)$ thus returns the point a projection oracle would have returned, and the method inherits the properties of a projection method---in particular Fejér monotonicity with respect to $\cX_\star$, that is $\norm{x_{k+1}-x_\star} \le \norm{x_k-x_\star}$ for every minimizer $x_\star$ (cf. \Cref{thm:lmo_is_projection}, part~\ref{it:lmoproj-distance})---while remaining implementable by a linear oracle.
This change of view is what yields the convergence guarantees and universal Hölder rates of \Cref{thm:localrate}.
That the half-space contains the solution set $\cX_\star$ is what drives our analysis, and it motivates a wider family of projection methods, to which \LLMO belongs and which we present next.

\subsection{A family of projection methods}
\label{sub:a_family_of_projection_methods}

We consider methods that, at a point $x_k \in \cX\setminus\cX_\star$, intersect the feasible set $\cX$ with a half-space whose outward normal is the gradient and then project onto the intersection.
Since $\nabla f(x) \neq 0$ for all $x \in \cX \setminus \cX_\star$ (\Cref{lem:zero_grad}), such a half-space may be described by its depth, that is, the distance from the current iterate to its boundary.
For $x \in \R^d, g \neq 0$ and $s \ge 0$, we define the half-space
\begin{equation}\label{eq:halfspace}
    \cH_x(g,s) \ \eqdef \ \curlybr{z\in\R^d\st\inner{\tfrac{g}{\norm{g}}}{x-z}\ \ge\ s} ,
\end{equation}
so that $\cH_x(g,s)$ has outward normal $g$ and lies at distance $s$ from the point $x$.
For fixed $x$ and $g$ it shrinks as $s$ grows, $\cH_x(g,s') \subseteq \cH_x(g,s)$ whenever $s \le s'$, so a larger $s$ means a deeper cut and a smaller set to project onto.

The family of projection methods we consider forms, at the current iterate $x_k$, the half-space
\begin{equation}\label{eq:halfspace_k}
    \cH_k \ \eqdef \ \cH_{x_k}\rbr{\nabla f(x_k),\ s_k}
\end{equation}
at an admissible (see \eqref{eq:admissible_cuts}) depth $s_k$, and moves by projecting onto the localized feasible set $\cX\cap\cH_k$, with $\Proj_\cC$ the Euclidean projection onto a closed and convex set $\cC$,
\begin{equation}\label{eq:proj_family}
    x_{k+1} = \Proj_{\cX \cap \cH_k}(x_k) .
\end{equation}
We require the cut to exclude the current iterate while retaining every minimizer,
\begin{equation}\label{eq:localization}
    x_k \notin \cH_k, \qquad \cX_\star \subseteq \cH_k .
\end{equation}

The first requirement, $x_k \notin \cH_k$, is equivalent to positivity of the chosen depth $s_k$.
The second, $\cX_\star \subseteq \cH_k$, asks every minimizer to lie at least $s_k$ beyond $x_k$ along the negative gradient direction, and so holds if and only if $s_k$ does not exceed
\begin{equation}
    \label{def:s-star}
    s_\star(x) \ \eqdef \ \inf_{x_\star \in \cX_\star} \inner{\frac{\nabla f(x)}{\norm{\nabla f(x)}}}{x-x_\star}
\end{equation}
at $x = x_k$ (\Cref{thm:alpha_star_exact}), the depth of the deepest localizing cut.
The admissible depths at $x_k$ therefore form the interval
\begin{equation}
    \label{eq:admissible_cuts}
    s_k \ \in \ \left(\,0 ,\ s_\star(x_k)\,\right] .
\end{equation}

Convexity of the objective ensures that this interval is always non-empty.
For every $x_\star \in \cX_\star$, convexity gives $\inner{\nabla f(x)}{x-x_\star} \ge f(x)-f_\star$, so that
\begin{equation}
    \label{eq:cvx-lower-bound}
    \inner{\frac{\nabla f(x)}{\norm{\nabla f(x)}}}{x-x_\star} \ \ge \ \frac{f(x)-f_\star}{\norm{\nabla f(x)}} \ = \ t(x) , \qquad \forall x_\star \in \cX_\star ,
\end{equation}
and therefore $s_\star(x) \ge t(x) > 0$.
The Polyak radius is itself an admissible depth, and we refer to the cut at that depth as the \emph{Polyak cut}.
It is the constrained counterpart of the classical subgradient cut \citep{polyak1969minimization,censor1982cyclic}: 
when $\cX = \R^d$, projecting $x$ onto the Polyak cut returns $x - \tfrac{f(x)-f_\star}{\norm{\nabla f(x)}^2}\nabla f(x)$, the gradient step with the Polyak stepsize.
Those admissible depths that are at least the Polyak radius, that is $s \in [\,t(x),\ s_\star(x)\,]$, form what we call the \emph{deep-cut band}.

\subsection{Contributions}
\label{sub:contributions}

The present work makes the following contributions.

\begin{itemize}
    \item \textbf{\algname{Local LMO} is a projection method (\Cref{thm:lmo_is_projection})}. 
    We show that for every $x \in \cX\setminus\cX_\star$ and every admissible radius $0 < t \le t(x)$, where $t(x)$ is the Polyak radius defined in \eqref{def:polyak-radius}, the \LLMO step $x_{\rm lmo}$ of \eqref{eq:local_lmo} taken at $x$ satisfies
    \begin{equation}
        \label{eq:contrib_identification}
        x_{\rm lmo} \ = \ \Proj_{\cX \cap \cH_x(\nabla f(x),\, s_{\rm lmo}(x,t))}(x) ,
    \end{equation}
    where the depth $s_{\rm lmo}(x,t) \eqdef \inner{\frac{\nabla f(x)}{\norm{\nabla f(x)}}}{x - x_{\rm lmo}}$ is admissible.
    In other words, \LLMO is secretly a projection method!

    \item \textbf{Universal Hölder rate for deep cuts (\Cref{thm:universal_band})}. 
    Every method of the form \eqref{eq:proj_family} whose depths $s_k$ lie in the deep-cut band $[\,t(x_k),\ s_\star(x_k)\,]$ drives its best iterate to optimality at rate $K^{-(1+\vartheta)/2}$ on the class of convex functions with $\vartheta$-Hölder continuous gradients, without knowledge of the exponent $\vartheta$ or the constant $L_\vartheta$.

    \item \textbf{Universal Hölder rate for \LLMO (\Cref{thm:localrate})}. 
    The \LLMO depth $s_{\rm lmo}(x_k,t_k)$ never exceeds the Polyak radius and may fall below it, so the previous result does not cover \LLMO in general and we analyze it separately: 
    run at the Polyak radius $t_k = t(x_k)$, it attains the same rate when the constrained minimizers are also unconstrained ones, that is $\norm{\nabla f(x_\star)} = 0$ for every $x_\star \in \cX_\star$ (\Cref{lem:grad_const_on_Xstar}), and the rate $K^{-1/(2-\vartheta)}$ otherwise.
    Both exponents are at least the Frank--Wolfe exponent $\vartheta$ \citep{nesterov2018complexity} on the range $[0,1]$, and neither requires the feasible set $\cX$ to be bounded, which the Frank--Wolfe guarantee does.
\end{itemize}

\section{Related Literature}
\label{sec:related_literature}

\paragraph{Frank--Wolfe and projection-free methods.} 
The conditional gradient method \citep{FrankWolfe1956,levitin1966constrained} accesses $\cX$ only through the global LMO \eqref{eq:lmo} and attains the classical $\cO(L\diam^2(\cX)/K)$ rate for $L$-smooth objectives \citep{Hazan2008Sparse,Jaggi2013,braun2022conditional}, which cannot be improved by methods restricted to that oracle \citep{lan2013complexity}.
Under $\vartheta$-Hölder-smoothness the guarantee degrades to $\cO(K^{-\vartheta})$ \citep{nesterov2018complexity} and requires $\cX$ to be bounded, as the global LMO may otherwise be unbounded below.
Away-step, pairwise and fully corrective variants \citep{wolfe1970integer,lacoste2015global} address the zig-zagging that underlies the $\cO(1/K)$ bound but keep the global oracle. 
For the plain method, linear rates require additional structure, such as a solution interior to $\cX$ \citep{guelat1986some}; 
other variants attain them over polytopes instead, at the price of a geometric constant of the feasible set \citep{lacoste2015global}.
\citet{hazan2026linearly} avoids projections for smooth convex sets by a supporting-tangent computation rather than a restricted oracle, and \citet{martinezrubio2026beyond} revisit the step size itself; that linear minimization and projection are not interchangeable in cost is made precise by \citet{combettes2021complexity}.
Interest in LMO-based methods has been renewed by deep-learning optimizers such as \algname{Muon} \citep{jordan2024muon}, whose update may be read as a linear minimization over a spectral-norm ball, a view developed in \algname{Scion} \citep{pethick2025training} and \algname{Gluon} \citep{riabinin2025gluon}.

\paragraph{Projection onto a separating cut.} 
The projection step in \eqref{eq:proj_family} has a long history.
For $\cX = \R^d$ it is the subgradient projector of \citet{polyak1969minimization}, taken up for convex feasibility problems by \citet{censor1982cyclic} and given efficiency estimates by \citet{kiwiel1996efficiency}.
Projections onto intersections of half-spaces of this kind underlie the weak-to-strong convergence principle of \citet{bauschke2001weak}, and the form $\Proj_{\cX\cap\cH_k}(x_k)$ appears in the variational-inequality method of \citet{solodov1999new}.
Closest to the geometry studied here are the level method of \citet{Lemarechal1995Variants}, which with a single cut and level $f_\star$ specializes to the projection onto the Polyak cut, 
the ballstep subgradient level methods of \citet{kiwiel1999ballstep}, which already combine a trust region with a level cut, 
and the minorant method of \citet{devanathan2024polyak}.
Fejér monotonicity with respect to $\cX_\star$ is the standard engine for this family \citep{bauschke1996projection,combettes2001quasi} and underlies our analysis as well.
What seems absent from the literature is the observation that, for an arbitrary closed convex $\cX$, a projection of this kind---onto a cut the oracle itself selects---is produced by a single linear minimization over a trust region, which is the content of \Cref{thm:lmo_is_projection}.

\paragraph{Universal and Hölder-smooth methods.} 
\citet{nesterov2015universal} introduced universal gradient methods, attaining $\cO(K^{-(1+\vartheta)/2})$ unaccelerated and $\cO(K^{-(1+3\vartheta)/2})$ accelerated across the whole Hölder scale without knowledge of $(\vartheta,L_\vartheta)$, by means of a backtracking line search.
The accelerated exponent is optimal and is attained without line search by the bundle-level methods of \citet{lan2015bundle}, at the cost of a compact feasible set and an auxiliary level subproblem per iteration, and by the auto-conditioned method of \citet{li2026simple}, which dispenses with compactness (see also \citet{deng2024uniformly}).
Among projection-free methods, \citet{ito2023parameter} remove the need for $(\vartheta,L_\vartheta)$ at the Frank--Wolfe exponent $\vartheta$ through an adaptive line search on a bounded domain; 
the exponent $\tfrac{1}{2-\vartheta}$ of \Cref{thm:localrate} exceeds $\vartheta$ on $[0,1)$ and requires neither a line search nor a bounded domain, at the price of knowing $f_\star$.

\paragraph{Local LMO and the Polyak stepsize.} 
\citet{richtarik2026local} reintroduce the local oracle \eqref{eq:local_lmo}, establish well-posedness, give closed forms for several structured feasible sets, and prove convergence under smoothness, strong convexity and bounded gradients, with the Fejér decrease $\norm{x_{k+1}-x_\star}^2 \le \norm{x_k-x_\star}^2 - t_k^2$ driving the analysis.
The oracle is the subproblem of the trust-region successive linear programming scheme analyzed by \citet{ferris1996linear}, and reappears as the local linear optimization oracle of \citet{garber2016linearly}, which relaxes \eqref{eq:local_lmo} by returning a point within a multiple of the radius and delivers linear convergence over polytopes for smooth, strongly convex objectives.
The trust-region viewpoint is shared with the ball-proximal point method of \citet{gruntkowska2025ball}, which minimizes the objective itself over a ball, and with its non-Euclidean counterpart \citep{gruntkowska2025non}, in which the linearized ball step studied here appears as the practical surrogate for the exact ball-proximal update.
The radius $t(x)$ is the constrained counterpart of the Polyak stepsize \citep{polyak1969minimization}, which is known to adapt to problem regularity: 
\citet{hazan2019revisiting} show that a single Polyak-type rule is near-optimal simultaneously across ranges of strong convexity and smoothness, \citet{he2025new} give tight analyses together with guarantees under Hölder smoothness and Hölder growth, and \citet{orabona2025new} recover the rule as gradient descent on a surrogate objective.
Like the Polyak stepsize, $t(x)$ presumes knowledge of $f_\star$, which may be estimated online \citep{abdukhakimov2025polyak}.

\section{Theory}
\label{sec:theory}

We say that a radius is admissible at $x$ if $0 < t \le t(x)$, where $t(x)$ denotes the Polyak radius defined in \eqref{def:polyak-radius}. 
Throughout, $g_k \eqdef \nabla f(x_k)$ and $a_k \eqdef f(x_k)-f_\star$ denote the gradient and the suboptimality gap at the $k$-th iterate, respectively, and $R \eqdef \dist(x_0,\cX_\star)$ denotes the initial distance to the solution set.

This section presents three results:
\Cref{thm:lmo_is_projection} identifies \LLMO as a member of the projection family \eqref{eq:proj_family}, and thus places every method we consider on one axis, the depth of the cut.
\Cref{thm:universal_band} shows that cutting at least as deep as the Polyak cut, that is $s_k \in [\,t(x_k), s_\star(x_k)\,]$, is sufficient to obtain a universal rate.
The \LLMO depth never exceeds $t(x_k)$, lying at or below the shallow endpoint of that band, and so \Cref{thm:universal_band} does not cover \LLMO in general; 
\Cref{thm:localrate} recovers a universal rate for it through a refined argument.
We defer the detailed proofs to the appendix, offering sketches of the key arguments here.

\subsection{Local LMO is a projection method}

Fix $x \in \cX \setminus \cX_\star$ and an admissible radius $t \in (0,t(x)]$, and let $x_{\rm lmo}$ solve \eqref{eq:local_lmo}. 
We define the \LLMO depth
\begin{equation}
    \label{def:s-lmo}
    s_{\rm lmo}(x,t) \ \eqdef \ \inner{\frac{\nabla f(x)}{\norm{\nabla f(x)}}}{x-x_{\rm lmo}} ,
\end{equation}
that is, the depth of the cut whose boundary passes through the point returned by the oracle.
We then obtain the following.

\begin{restatable}{theorem}{lmoprojthm}
\label{thm:lmo_is_projection}
    Fix $x \in \cX \setminus \cX_\star$ and an admissible radius $t \in (0, t(x)]$, and let $x_{\rm lmo} \in \argmin_{z \in \cX \cap\cB(x,t)} \inner{\nabla f(x)}{z}$ be a \LLMO step.
    Then:
    \begin{enumerate}
        \item \label{it:lmoproj-localizes}
        The \LLMO depth is admissible and does not exceed the radius, that is,
        \begin{equation*}
            0 \ < \ s_{\rm lmo}(x,t) \ \le \ t \ \le \ t(x) \ \le \ s_\star(x) ,
        \end{equation*}
        so the half-space $\cH_x\rbr{\nabla f(x), s_{\rm lmo}(x,t)}$ separates $x$ from, and localizes, the set of minimizers:
        \begin{equation*}
            x \notin \cH_x\rbr{\nabla f(x), s_{\rm lmo}(x,t)} , \qquad \cX_\star \subseteq \cH_x\rbr{\nabla f(x), s_{\rm lmo}(x,t)} .
        \end{equation*}

        \item \label{it:lmoproj-projects}
        The projection onto the localized feasible set \eqref{eq:proj_family} coincides with the \LLMO step, that is, 
        \begin{equation*}
            x_{\rm lmo} = \Proj_{\cX \cap \cH_x(\nabla f(x),\, s_{\rm lmo}(x,t))}(x) .
        \end{equation*}
        In particular, $\argmin_{z \in \cX \cap\cB(x,t)} \inner{\nabla f(x)}{z}$ is a singleton.

        \item \label{it:lmoproj-stepsize}
        The method takes a step of size $t$, that is,
        \begin{equation*}
            \norm{x_{\rm lmo} - x} = \dist\rbr{x, \cX \cap \cH_x(\nabla f(x),\, s_{\rm lmo}(x,t))} = t.
        \end{equation*}

        \item \label{it:lmoproj-distance}
        The squared distance of the new point to any minimizer is reduced by at least $t^2$, that is,
        \begin{equation*}
            \norm{x_{\rm lmo} - x_\star}^2 \le \norm{x - x_\star}^2 - t^2 , \quad \forall x_\star \in \cX_\star .
        \end{equation*}
    \end{enumerate}
\end{restatable}

\begin{proof}[Proof sketch]
    We write $N_\cC(z)$ for the normal cone of a convex set $\cC$ at a point $z \in \cC$.
    The first-order optimality condition for the local subproblem \eqref{eq:local_lmo} may be restated as $0 \in \nabla f(x)+N_{\cX}(x_{\rm lmo})+N_{\cB(x,t)}(x_{\rm lmo})$. 
    Since the normal cone of a ball is $\curlybr{0}$ at an interior point and the ray through $x_{\rm lmo}-x$ at a boundary point (\Cref{lem:normal_cone_ball}), there exist $n \in N_{\cX}(x_{\rm lmo})$ and $\lambda \ge 0$ with
    \begin{equation}
        \label{eq:sketch_kkt}
        \nabla f(x)+n+\lambda\rbr{x_{\rm lmo}-x} \ = \ 0 .
    \end{equation}
    Parts~\ref{it:lmoproj-stepsize} and~\ref{it:lmoproj-projects} both follow from \eqref{eq:sketch_kkt}, in that order, each by a case split on whether the multiplier $\lambda$ vanishes; 
    parts~\ref{it:lmoproj-localizes} and~\ref{it:lmoproj-distance} are then consequences.
    Admissibility is used in the form
    \begin{equation}
        \label{eq:sketch_admissible}
        t \ \le \ t(x) \ = \ \frac{f(x)-f_\star}{\norm{\nabla f(x)}} \ \le \ \inner{\frac{\nabla f(x)}{\norm{\nabla f(x)}}}{x-x_\star} , \qquad \forall x_\star \in \cX_\star ,
    \end{equation}
    which is first-order convexity divided by $\norm{\nabla f(x)}$.

    Suppose $\lambda > 0$. 
    Since the step is non-zero (\Cref{lem:local_fw_gap}), $\lambda(x_{\rm lmo}-x)$ is a non-zero element of $N_{\cB(x,t)}(x_{\rm lmo})$, and that cone is $\curlybr{0}$ at every interior point of the ball (\Cref{lem:normal_cone_ball}).
    Hence, $x_{\rm lmo}$ lies on the sphere and the step has length $t$. 
    If instead $\lambda = 0$, then $-\nabla f(x) \in N_{\cX}(x_{\rm lmo})$, so $\inner{\nabla f(x)}{x_{\rm lmo}-x_\star} \le 0$ for every minimizer $x_\star$, while \eqref{eq:sketch_admissible}, Cauchy--Schwarz and $\norm{x_{\rm lmo}-x} \le t$ give the reverse inequality. 
    The resulting equality forces equality both in the Cauchy--Schwarz bound $\inner{\nabla f(x)}{x_{\rm lmo}-x} \ge -\norm{\nabla f(x)}\,\norm{x_{\rm lmo}-x}$ and in $\norm{x_{\rm lmo}-x} \le t$, so that
    \begin{equation}
        \label{eq:sketch_antiparallel}
        x-x_{\rm lmo} \ = \ t\,\frac{\nabla f(x)}{\norm{\nabla f(x)}} , \qquad\text{and in particular}\qquad \norm{x_{\rm lmo}-x} \ = \ t ,
    \end{equation}
    and the step has again length $t$, establishing part~\ref{it:lmoproj-stepsize}.

    For part~\ref{it:lmoproj-projects}, we must establish $\inner{x-x_{\rm lmo}}{y-x_{\rm lmo}} \le 0$ for every $y$ in the localized feasible set, this being the variational characterization of the projection. 
    First, $\inner{n}{y-x_{\rm lmo}} \le 0$, because $y \in \cX$ and $n \in N_{\cX}(x_{\rm lmo})$. 
    Second,
    \begin{equation}
        \label{eq:sketch_membership}
        y \in \cH_x\rbr{\nabla f(x),\, s_{\rm lmo}(x,t)} \qquad\Longleftrightarrow\qquad \inner{\nabla f(x)}{y-x_{\rm lmo}} \ \le \ 0 ,
    \end{equation}
    since inserting the depth \eqref{def:s-lmo} into the definition \eqref{eq:halfspace} of the half-space reads $\inner{\nabla f(x)}{x-y} \ge \inner{\nabla f(x)}{x-x_{\rm lmo}}$.
    For $\lambda > 0$, taking the inner product of \eqref{eq:sketch_kkt} with $y-x_{\rm lmo}$ and discarding the non-positive terms leaves $\lambda\inner{x_{\rm lmo}-x}{y-x_{\rm lmo}} \ge 0$, which is the claim. 
    For $\lambda = 0$, \eqref{eq:sketch_antiparallel} gives $x-x_{\rm lmo}$ as a positive multiple of $\nabla f(x)$, and the claim is then \eqref{eq:sketch_membership} again.

    In part~\ref{it:lmoproj-localizes}, $s_{\rm lmo}(x,t) \le t$ is Cauchy--Schwarz applied to the step $x-x_{\rm lmo}$, whose length is $t$ by part~\ref{it:lmoproj-stepsize}, and $s_{\rm lmo}(x,t) > 0$ holds because the feasible point $(1-\mu)x+\mu x_\star$ with $\mu = t/\norm{x-x_\star}$ strictly improves the linear model (\Cref{lem:local_fw_gap}).
    Taking the infimum over $x_\star$ in \eqref{eq:sketch_admissible} gives the remaining inequalities $t \le t(x) \le s_\star(x)$, which is what makes the cut localize $\cX_\star$ (\Cref{thm:alpha_star_exact}).
    Part~\ref{it:lmoproj-distance} is the projection inequality of part~\ref{it:lmoproj-projects} evaluated at a minimizer, which lies in the localized feasible set by part~\ref{it:lmoproj-localizes}, together with the step length (\Cref{thm:projection_fejer_halfspace}).
\end{proof}

The step length in part \ref{it:lmoproj-stepsize} and the decrease in part \ref{it:lmoproj-distance} are Theorem 3.1(ii)--(iii) of \citet{richtarik2026local}; 
we rederive them because the projection identity yields them at no extra cost and because \Cref{thm:sharp_fejer} later sharpens part \ref{it:lmoproj-distance}. 
Parts \ref{it:lmoproj-localizes} and \ref{it:lmoproj-projects}, and the identification of that length with the distance to the localized feasible set, are novel insights.
\Cref{thm:lmo_is_projection} places \LLMO inside the family \eqref{eq:proj_family} and shows that \LLMO is (secretly) a projection method!
Two parts carry the analysis that follows: part~\ref{it:lmoproj-distance} is Fejér monotonicity (\Cref{thm:projection_fejer_halfspace}), which makes the squared step lengths summable, and part~\ref{it:lmoproj-localizes} places \LLMO on the shallow side of the Polyak depth via $s_{\rm lmo}(x,t) \le t \le t(x)$.
The next subsection turns to the members that cut at least as deep as the Polyak cut.

\subsection{A universal rate for deep cuts}

The rates in this subsection and the next are stated on the Hölder scale.
We say that the gradient map $\nabla f$ of the objective function is Hölder continuous with exponent $\vartheta\in[0,1]$ and constant $L_\vartheta$ (or simply, $\vartheta$-Hölder continuous) if the Bregman divergence $D_f(u,v) \eqdef f(u)-f(v)-\inner{\nabla f(v)}{u-v}$ satisfies
\begin{equation}\label{eq:holder}
    D_f(u,v)\ \le\ \tfrac{L_\vartheta}{1+\vartheta}\norm{u-v}^{1+\vartheta}, \qquad\forall u,v\in\R^d .
\end{equation}
The exponent $\vartheta = 1$ corresponds to the smooth case, $\vartheta = 0$ to that of bounded gradient variation, and the scale interpolates between them.
Following \citet{nesterov2015universal}, we call a method \emph{universal} if it needs no a priori knowledge of $(\vartheta,L_\vartheta)$; 
the methods below are universal in this sense.

For the subfamily of projection methods \eqref{eq:proj_family} in the deep-cut band $[t(x), s_\star(x)]$, we obtain the following.

\begin{restatable}[Universal Hölder rate on the deep-cut band]{theorem}{universalbandthm}
\label{thm:universal_band}
    Assume $\nabla f$ is Hölder continuous with exponent $\vartheta\in[0,1]$ and constant $L_\vartheta$ as in \eqref{eq:holder}. 
    Let $\{x_k\}$ be generated from $x_0\in\cX\setminus\cX_\star$ by $x_{k+1} = \Proj_{\cX \cap \cH_{x_k}(\nabla f(x_k),\, s_k)}(x_k)$ with any (possibly adaptive) depths in the deep-cut band, that is, satisfying $t(x_k) \le s_k \le s_\star(x_k)$ for all $k$, the run stopping if some $x_k \in \cX_\star$.
    Then, for every $K \ge 1$, we have
    \begin{equation}
    \label{eq:band_rate}
        \min_{0\le k\le K} \bigl(f(x_k)-f_\star\bigr) \le \frac{L_\vartheta}{1+\vartheta} \left(\frac{\dist^2(x_0,\cX_\star)}{K}\right)^{\frac{1+\vartheta}{2}} .
    \end{equation}
\end{restatable}

\begin{proof}[Proof sketch]
    The argument rests on three parts, and the Hölder data enter only in the last.
    First, cutting at least as deep as the Polyak cut:
    If $x_{k+1}$ lies in the Polyak cut at $x_k$, then the linearization of $f$ at $x_k$, evaluated at $x_{k+1}$, is at most $f_\star$. 
    Subtracting it from $f(x_{k+1})$ leaves
    \begin{equation}
        \label{eq:sketch_polyak_cut}
        f(x_{k+1})-f_\star \ \le \ D_f(x_{k+1},x_k) ,
    \end{equation}
    so the function gap may be upper-bounded by the Bregman divergence of a single step (\Cref{lem:polyak_cut}).
    Second, Fej\'er monotonicity (\Cref{thm:projection_fejer_halfspace}): 
    any cut retaining $\cX_\star$ gives $\norm{x_{k+1}-x_\star}^2 \le \norm{x_k-x_\star}^2-\norm{x_{k+1}-x_k}^2$, so the squared step lengths sum to at most $R^2$ and the shortest of $K$ steps obeys $\min_{k<K}\norm{x_{k+1}-x_k}^2 \le R^2/K$.
    Third, Hölder continuity is invoked once, at that short step: 
    $D_f(x_{k+1},x_k) \le \tfrac{L_\vartheta}{1+\vartheta}\norm{x_{k+1}-x_k}^{1+\vartheta}$. 
    Combining the three gives \eqref{eq:band_rate}.
    Neither $\vartheta$ nor $L_\vartheta$ is used to generate the iterates; 
    they appear only in this last line, which is what makes the rate universal.
\end{proof}

The family of projection methods described in \Cref{thm:universal_band} attains the rate $K^{-(1+\vartheta)/2}$ across the whole Hölder scale, the rate of the universal primal gradient method of \citet{nesterov2015universal}. 
There, universality is obtained by a backtracking search for a local Hölder constant;
here, it comes from the geometry of the cut instead.

\subsection{A universal rate for Local LMO}

By part~\ref{it:lmoproj-localizes} of \Cref{thm:lmo_is_projection}, \LLMO cuts at depth at most $t(x_k)$, that is, at or below the shallow endpoint of the deep-cut band from \Cref{thm:universal_band}, and is therefore not covered by the previous result in general.

\begin{restatable}[Universal Hölder rate for \LLMO]{theorem}{localratethm}
\label{thm:localrate}
    Let $\nabla f$ be Hölder continuous with exponent $\vartheta \in [0,1]$ and constant $L_\vartheta$, and let $\{x_k\}$ be generated by \LLMO at the Polyak radius $t_k=t(x_k)$ from $x_0 \in \cX \setminus \cX_\star$, the run stopping if some $x_k \in \cX_\star$.
    Then, for all $K \ge 1$, we have
    \begin{equation}\label{eq:localrate}
        \min_{0\le k\le K}\bigl(f(x_k)-f_\star\bigr)\ \le\
        \max\left\{\ \Lambda_1\,K^{-\frac{1+\vartheta}{2}}\ ,\ \ \Lambda_2\,K^{-\frac{1}{2-\vartheta}}\ \right\},
    \end{equation}
    where $R = \dist(x_0,\cX_\star)$ and $G_\star = \norm{\nabla f(x_\star)}$ for any $x_\star \in \cX_\star$, and
    \begin{equation}\label{eq:lambdas}
        \Lambda_1\eqdef \tfrac{25}{2}\,L_\vartheta R^{1+\vartheta},
        \qquad
        \Lambda_2\eqdef 25 \max\Bigl\{\bigl(L_\vartheta R^2 G_\star^{1-\vartheta}\bigr)^{\frac1{2-\vartheta}},\ G_\star R\Bigr\} .
    \end{equation}
\end{restatable}

\begin{proof}[Proof sketch]
    The first ingredient of the previous argument is unavailable: 
    the \LLMO cut is no deeper than the Polyak cut, so $x_{k+1}$ need not lie in the latter and \eqref{eq:sketch_polyak_cut} may fail. 
    The shortfall in depth, weighted by the gradient norm, is what we call the overshoot,
    \begin{equation}
        \label{def:overshoot}
        \Gamma_k \ \eqdef \ \norm{g_k}\rbr{t(x_k)-s_{\rm lmo}(x_k,t_k)} \ \ge \ 0 ,
    \end{equation}
    which vanishes precisely when the \LLMO cut is the Polyak cut (\Cref{lem:overshoot}).
    The exact identity $f(x_{k+1})-f_\star = D_f(x_{k+1},x_k) + \Gamma_k$ shows that the amount by which \eqref{eq:sketch_polyak_cut} fails is exactly the overshoot. 
    Since $D_f \ge 0$, no \LLMO step can push the suboptimality below $\Gamma_k$, however smooth $f$ may be (\Cref{lem:overshoot_lower_bound}).

    The quantity that our analysis rests on is, with $\ell_k \eqdef \inner{g_k}{x_k-x_{k+1}}$ the decrease of the linear model,
    \begin{equation}
        \label{def:rho}
        \rho_k \ \eqdef \ \frac{\ell_k}{a_k} \ \in \ (0,1] ,
    \end{equation}
    the fraction of the first-order decrease available at $x_k$ that the oracle actually captures; 
    since $\ell_k = \norm{g_k}s_{\rm lmo}(x_k,t_k)$ and $a_k = \norm{g_k}t(x_k)$, the range $(0,1]$ is part~\ref{it:lmoproj-localizes} of \Cref{thm:lmo_is_projection} and $\Gamma_k = (1-\rho_k)a_k$.
    At the Polyak radius, $\rho_k$ is the cosine of the angle between $g_k$ and the step $x_k-x_{k+1}$: 
    it equals $1$ when the step runs straight down the gradient and tends to $0$ when the feasible set forces the step almost orthogonal to $g_k$.

    A small $\rho_k$ is bad for the function value---the floor $\Gamma_k$ leaves $a_{k+1} \ge (1-\rho_k)a_k$, so one step removes at most a $\rho_k$ fraction of the gap, while the descent inequality delivers only $a_{k+1} \le a_k(1-\rho_k+\delta_k)$, with $\delta_k \eqdef \tfrac{L_\vartheta}{1+\vartheta}t_k^{1+\vartheta}/a_k$ the Hölder error---but it is good for the distance.
    The projection identity of \Cref{thm:lmo_is_projection} sharpens Fej\'er monotonicity to
    \begin{equation}
        \label{eq:sketch_sharp_fejer}
        \norm{x_{k+1}-x_\star}^2 \ \le \ \norm{x_k-x_\star}^2-t_k^2\,\frac{2-\rho_k}{\rho_k}
    \end{equation}
    (\Cref{thm:sharp_fejer}), in which the overshoot amplifies the contraction by the factor $(2-\rho_k)/\rho_k \ge 1/\rho_k$: a nearly tangential step barely lowers $f$, but pulls the iterate closer to $\cX_\star$ than its length alone would suggest.
    Summing gives the refined budget $\sum_{k<K}t_k^2/\rho_k \le R^2$.
    The two fit together because
    \begin{equation}
        \label{eq:sketch_gaincost}
        \underbrace{\frac{\rho_k}{a_k^2}}_{\text{gain in }\Phi_k \eqdef a_k^{-2}} \cdot \underbrace{\frac{t_k^2}{\rho_k}}_{\text{Fej\'er cost}} \ = \ \frac{t_k^2}{a_k^2} \ = \ \frac{1}{\norm{g_k}^2}
    \end{equation}
    does not depend on $\rho_k$ at all: 
    a step that makes little progress in function value pays for it in the Fej\'er budget by exactly the reciprocal factor, so no choice of $\rho_k$ is free.
    
    The remainder of the proof rests on an accounting argument.
    We write $\bar a$ for the level below which $G_\star$ dominates the Hölder part of the gradient bound; 
    above it the refined budget alone delivers the exponent $\tfrac{1+\vartheta}{2}$.
    Below $\bar a$, a two-sided gradient bound pins $\norm{g_k}$ to within a constant factor of $G_\star$, and the $K$ steps are counted in three groups: 
    those above $\bar a$ are few by the budget; 
    those that stall, in the sense $\rho_k < 2\delta_k$, are few because each consumes a fixed share of the budget; 
    and the productive remainder is counted by Cauchy--Schwarz against $\Phi_k$ using \eqref{eq:sketch_gaincost}, a count needed only once $K$ is large. 
    This produces the exponent $\tfrac1{2-\vartheta}$, and \eqref{eq:localrate} is the larger of the two bounds.
\end{proof}

Two remarks on \eqref{eq:localrate} are in order.
First, the second term disappears when the gradient vanishes at the solution: 
if $G_\star = 0$, so that every constrained optimum is also an unconstrained one, then \eqref{eq:localrate} holds with $\Lambda_2$ replaced by $0$, and \LLMO attains the exponent $\tfrac{1+\vartheta}{2}$ of \Cref{thm:universal_band}, differing only in the constant, although its cut is never deeper than the shallow endpoint of the band.

Second, at $\vartheta=0$ the two exponents in \eqref{eq:localrate} coincide, at $\tfrac12$, and a direct argument improves the constant: 
a Hölder-$0$ bound forces $\norm{g_k}\le G_\star+L_0$ along the whole trajectory, and Fej\'er monotonicity alone then gives
\begin{equation}
    \min_{0\le k\le K} \bigl(f(x_k)-f_\star\bigr)\ \le\ \frac{\rbr{G_\star+L_0}R}{\sqrt K} ,
\end{equation}
the same rate as \eqref{eq:localrate} with a smaller constant (\Cref{prop:rate_zero}).

\section{Experiments}
\label{sec:experiments}

To illustrate the qualitative difference between \LLMO and the deep-cut methods, we consider the two-dimensional objective function
\begin{equation}
    \label{eq:2d-example}
    f(x,y) = \frac{L}{1 + \vartheta} \abs{x}^{1 + \vartheta} + G y
\end{equation}
to be optimized over the half-space $\cX = \curlybr{(x,y) \st y \ge 0}$,
where $G \ge 0$ and $L > 0$ are constants and $\vartheta \in (0,1]$ governs Hölder continuity as defined in \eqref{eq:holder}; 
the endpoint $\vartheta = 0$ is excluded to ensure differentiability of $f$.

For $G > 0$, the minimizer is unique, $x_\star = (0,0)$, and lies on the boundary of the feasible set $\cX$; 
for $G = 0$ the solution set is the ray $\curlybr{(0,y) \st y \ge 0}$, of which $(0,0)$ is the point nearest to the initial point $x_0 = (R,0)$, for some $R > 0$, located on the boundary of $\cX$.
For $G > 0$ the affine part of the objective pushes against the constraint $y \ge 0$, and for every $G \ge 0$ we have $\nabla f(x_\star) = (0,G)$, so $G_\star = \norm{\nabla f(x_\star)} = G$.
The $\vartheta$-Hölder continuous part has its Hölder singularity at $x = 0$, and \eqref{eq:holder} holds with $L_\vartheta = 2^{1-\vartheta} L$, independent of $G$;
varying $G$ therefore varies how constrained the optimizer is while $\vartheta$ and $L_\vartheta$ stay fixed.
The two coordinates thus separate the quantities governing \Cref{thm:universal_band,thm:localrate}: the Hölder term governs the growth of $f$ along the face $y=0$, on which the iterates travel, while the affine term contributes nothing to the value there.

We run \LLMO at the Polyak radius $t_k = t(x_k)$ and, for comparison, the projection method \eqref{eq:proj_family} at the Polyak cut $s_k = t(x_k)$, the shallow endpoint of the deep-cut band.
Both iterations remain on the face $y = 0$, so, writing $x_k = (u_k,0)$ with $u_k > 0$ and using $a_k = \tfrac{L}{1+\vartheta}u_k^{1+\vartheta}$ and $g_k = \rbr{Lu_k^{\vartheta},\, G}$, the iterates are available in closed form:
\begin{equation}
    \label{eq:2d-dynamics}
    u_{k+1}^{\rm lmo} \ = \ u_k - \frac{a_k}{\norm{g_k}} , \qquad \norm{g_k} = \sqrt{L^2u_k^{2\vartheta}+G^2} ,
    \qquad\text{whereas}\qquad u_{k+1}^{\rm proj} \ = \ \frac{\vartheta}{1+\vartheta}\,u_k .
\end{equation}
The \LLMO update is the step of length $t_k$ supplied by part~\ref{it:lmoproj-stepsize} of \Cref{thm:lmo_is_projection}, the oracle returning the point $(u_k-t_k,0)$ on the face for every $G \ge 0$; 
the projection update, on the other hand, does not involve $G$ at all.

\begin{figure}[t]
    \centering
    \includegraphics[width=0.5\textwidth]{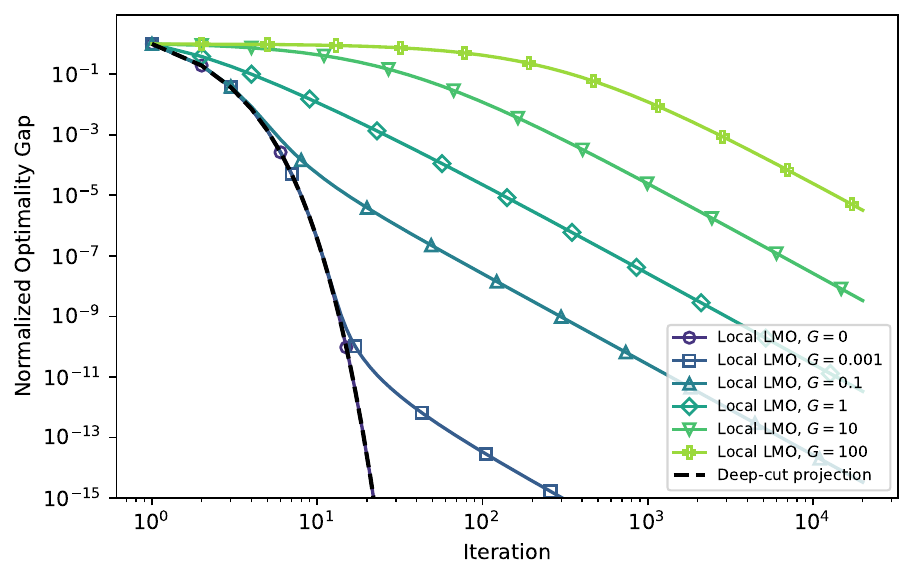}
    \caption{Normalized optimality gap $a_k/a_0$ on \eqref{eq:2d-example} with $\vartheta = 1/2$ and $L = R = 1$: \LLMO at the Polyak radius for a range of values of $G$, and the projection method \eqref{eq:proj_family} at the Polyak cut, which does not depend on $G$. \LLMO coincides with it at $G = 0$; for larger $G$, it first traverses a plateau whose length grows linearly in $G$.}
    \label{fig:stall}
\end{figure}

For large $G$, the \LLMO trajectory in \Cref{fig:stall} is essentially flat over a long initial phase whose length grows in proportion to $G$, while the projection method converges linearly and is unaffected.
The mechanism is best understood through \eqref{eq:2d-dynamics}: 
at a given position $u_k$ the suboptimality $a_k$ does not depend on $G$, because the affine term vanishes on the face, whereas the gradient norm does.
The Polyak radius $t_k = a_k/\norm{g_k} \le a_k/G$ is therefore throttled, and since the step length equals $t_k$ and $a_k$ decreases here, the iterate advances by at most $a_0/G$ per iteration. 
Covering the initial distance $R$ therefore takes at least $GR/a_0$ iterations, which is the plateau one observes in \Cref{fig:stall}.

In \eqref{eq:2d-example}, the \LLMO step moves along the face of the feasible set, so the fraction \eqref{def:rho} is
\begin{equation}
    \label{eq:2d-rho}
    \rho_k \ = \ \frac{Lu_k^{\vartheta}}{\sqrt{L^2u_k^{2\vartheta}+G^2}} \ = \ \frac{\inner{g_k}{x_k-x_{k+1}}}{\norm{g_k}\,\norm{x_k-x_{k+1}}} ,
\end{equation}
and a large $G$ tilts the gradient towards the normal of the constraint, driving $\rho_k$ to zero.
The projection method is not throttled in this way: 
it moves the full $\dist(x_k,\cX\cap\cH_k)$ to the Polyak cut, which here equals $t_k/\rho_k$, whereas \LLMO moves only $t_k = \dist(x_k,\cH_k)$.
This slack separates \Cref{thm:universal_band} from \Cref{thm:localrate}. 
It vanishes when $G = 0$, where the lowest curve of \Cref{fig:stall} shows the two methods coinciding: 
the gradient is parallel to the face, the overshoot is zero (\Cref{lem:overshoot}), and both updates reduce to $u_{k+1} = \tfrac{\vartheta}{1+\vartheta}u_k$.

\section{Conclusion}
\label{sec:conclusion}

We have shown that \LLMO---a single linear minimization over the feasible set intersected with a ball around the current iterate introduced as a projection-free approach---does, in fact, implement a projection.
As long as the ball radius does not exceed the Polyak radius $\rbr{f(x_k)-f_\star}/\norm{\nabla f(x_k)}$, the point the oracle returns is the Euclidean projection of $x_k$ onto the feasible set intersected with a half-space whose outward normal is the gradient and which separates $x_k$ from every minimizer.
\LLMO is therefore a projection method that happens to be implementable by a linear oracle, and the half-space it projects onto is one it selects itself.
\LLMO is part of the family of projection methods \eqref{eq:proj_family} which may be indexed by the depth of the half-spaces.

Cutting at least as deep as the Polyak cut, while still retaining every minimizer, is enough for the rate $K^{-(1+\vartheta)/2}$ on convex objectives with $\vartheta$-H\"older continuous gradient, and the guarantee holds for the whole family of such cuts.
\LLMO may cut shallower than the Polyak cut, and the shortfall in depth separates it from the deep-cut methods.
Charging the shortfall to the distance to the solution set rather than to the function value sharpens the Fej\'er contraction instead of destroying it, and yields a universal rate for \LLMO, as well: 
the same $K^{-(1+\vartheta)/2}$ when the gradient vanishes at the solution, and $K^{-1/(2-\vartheta)}$ otherwise.
Both exponents are at least the Frank--Wolfe exponent $\vartheta$ over the whole range $[0,1]$, and neither guarantee requires the feasible set to be bounded.

\begin{ack}
The research reported in this publication was supported by funding from King Abdullah University of Science and Technology (KAUST): i) KAUST Baseline Research Scheme, ii) CRG Grant ORFS-CRG12-2024-6460, and iii) Center of Excellence for Generative AI, under award number 5940.

The authors thank Jan Niklas Kolzenburg for fruitful discussions.
\end{ack}

\bibliographystyle{plainnat}
\bibliography{local-lmo-projection}

\clearpage
\appendix
\part*{Appendix}

\tableofcontents

\newpage

\section{Proof of \Cref{thm:lmo_is_projection}: \algname{Local LMO} is a projection method}
\label{app:identification}

\paragraph{Proof outline.}
\Cref{thm:lmo_is_projection} establishes that the \algname{Local LMO} step is the projection of $x$ onto $\cX \cap \cH(\nabla f(x), \alpha_{\rm lmo}(x,t))$, that its length is exactly $t$, and that the squared distance to any minimizer decreases by at least $t^2$.

Our proof of this result takes the following form.
\begin{enumerate}
    \item    
    \Cref{lem:zero_grad,lem:generic-decomposition,lem:normal_cone_ball} establish basic geometric properties that are used throughout this and the following sections.

    \item
    \Cref{thm:alpha_star_exact} shows that the set of minimizers, $\cX_\star$, lies in every sufficiently shallow half-space $\cH = \cH(\nabla f(x),\alpha)$ with $\alpha \ge \alpha_\star(x)$.
    From this, we establish Fejér monotonicity, a key ingredient for the remaining proof steps, for the whole family of projection methods of the form \eqref{eq:proj_family} in \Cref{thm:projection_fejer_halfspace}.

    \item
    \Cref{thm:alpha_star_exact,lem:alpha_t_vs_alpha_cvx,lem:alpha_t_le_alpha_lmo} show that $\alpha_\star(x) \le \alpha_{\rm cvx}(x) \le \alpha_{\rm proj}(x,t) \le \alpha_{\rm lmo}(x,t)$ for admissible $t$, with $\alpha_{\rm cvx} = \alpha_{\rm proj}$ when $t = t(x)$.

    \item
    \Cref{thm:lmo_boundary_and_projection} shows that the linear minimizer over $\cX \cap \cB(x,t)$ lies on the sphere $\norm{z-x} = t$, and that it satisfies the variational characterization of the projection of $x$ onto $\cX \cap \cH(\nabla f(x),\alpha_{\rm lmo}(x,t))$; that the distance from $x$ to this set equals $t$ is then a consequence.
\end{enumerate}

\subsection{Levels and depths}
\label{sub:levels_and_depths}

The main text indexes a cut by its depth, the distance from the current iterate to the cut, as in \eqref{eq:halfspace}.
The proofs below are more naturally written in terms of its level, the value taken by the linear functional $\inner{g}{\cdot}$ on the cut.
For $g \in \R^d$ and $\alpha \in \R$ we write
\begin{equation}
    \label{eq:halfspace_level}
    \cH(g,\alpha) \ \eqdef \ \curlybr{z\in\R^d \st \inner{g}{z} \le \alpha} ,
\end{equation}
so that, for $g \neq 0$ and $\alpha \le \inner{g}{x}$, the two descriptions are interchangeable through
\begin{equation}
    \label{eq:dictionary}
    \cH_x(g,s) \ = \ \cH\rbr{g,\ \inner{g}{x}-s\norm{g}} , \qquad\text{equivalently}\qquad \cH(g,\alpha) \ = \ \cH_x\rbr{g,\ \tfrac{\inner{g}{x}-\alpha}{\norm{g}}} .
\end{equation}
On that range the correspondence $\alpha \leftrightarrow s = \rbr{\inner{g}{x}-\alpha}/\norm{g}$ is an order-reversing affine bijection onto the depths $s \ge 0$: a larger level is a shallower cut and a larger half-space.
Four levels occur below, namely $\alpha_\star(x)$ of \eqref{eq:alpha_star_only_def}, $\alpha_{\rm cvx}(x)$ of \eqref{eq:alpha_cvx_def}, $\alpha_{\rm proj}(x,t)$ of \eqref{eq:alpha-proj-def} and $\alpha_{\rm lmo}(x,t)$ of \eqref{eq:alpha-lmo-def}.
For $x \in \cX\setminus\cX_\star$ they are the levels of the cuts at depths
\begin{equation}
    \label{eq:level_depth_table}
    \alpha_\star(x) \leftrightarrow s_\star(x) , \qquad \alpha_{\rm cvx}(x) \leftrightarrow t(x) , \qquad \alpha_{\rm proj}(x,t) \leftrightarrow t , \qquad \alpha_{\rm lmo}(x,t) \leftrightarrow s_{\rm lmo}(x,t) ,
\end{equation}
with $s_\star$ as in \eqref{def:s-star} and $s_{\rm lmo}$ as in \eqref{def:s-lmo}; 
the last of these is the identity $s_{\rm lmo}(x,t) = \bigl(\inner{\nabla f(x)}{x}-\alpha_{\rm lmo}(x,t)\bigr)/\norm{\nabla f(x)}$, which holds because $\alpha_{\rm lmo}(x,t) = \inner{\nabla f(x)}{x_{\rm lmo}}$.
Consequently the level chain $\alpha_\star(x) \le \alpha_{\rm cvx}(x) \le \alpha_{\rm proj}(x,t) \le \alpha_{\rm lmo}(x,t)$ established in \Cref{thm:alpha_star_exact,lem:alpha_t_vs_alpha_cvx,lem:alpha_t_le_alpha_lmo} is the same as the chain of depths $s_{\rm lmo}(x,t) \le t \le t(x) \le s_\star(x)$ asserted in part~\ref{it:lmoproj-localizes} of \Cref{thm:lmo_is_projection}.
In the same way, a cut localizes $\cX_\star$ if and only if $s \le s_\star(x)$, equivalently $\alpha \ge \alpha_\star(x)$; 
and it lies in the deep-cut band, $t(x) \le s \le s_\star(x)$, if and only if $\alpha_\star(x) \le \alpha \le \alpha_{\rm cvx}(x)$.

\subsection{Preliminary lemmas}
\label{sub:preliminaries}

\begin{lemma}
\label{lem:zero_grad}
    If $x \in \cX \setminus \cX_\star$, then $\nabla f(x) \neq 0$.
\end{lemma}

\begin{proof}
    We show the contrapositive.
    Suppose $x \in \cX$ and $\nabla f(x) = 0$.
    Then, by convexity and differentiability of $f$ (\Cref{ass:main}), we have
    \begin{equation*}
        f(y) \ge f(x) + \inner{\nabla f(x)}{y-x} = f(x), \qquad \forall y \in \R^d.
    \end{equation*}
    Thus, $x$ is a global minimizer of $f$, and, in particular, a minimizer over the feasible set $\cX$.
    Hence, $x \in \cX_\star$.
\end{proof}

\begin{lemma}
\label{lem:grad_const_on_Xstar}
    The gradient of $f$ is constant on the solution set: $\nabla f(u) = \nabla f(v)$ for all $u,v \in \cX_\star$.
    In particular, $G_\star \eqdef \norm{\nabla f(x_\star)}$ does not depend on the choice of $x_\star \in \cX_\star$.
\end{lemma}

\begin{proof}
    Let $u,v \in \cX_\star$.
    First-order optimality of $v$ for $f$ over $\cX$ gives $\inner{\nabla f(v)}{u-v} \ge 0$, while $f(u) = f(v) = f_\star$, so
    \begin{equation*}
        D_f(u,v) \ = \ f(u)-f(v)-\inner{\nabla f(v)}{u-v} \ = \ -\inner{\nabla f(v)}{u-v} \ \le \ 0 .
    \end{equation*}
    Convexity of $f$ gives $D_f(u,v) \ge 0$, hence $D_f(u,v) = 0$, that is $\inner{\nabla f(v)}{u-v} = 0$.
    Therefore, for every $w \in \R^d$, convexity at $v$ yields
    \begin{equation*}
        f(w) \ \ge \ f(v)+\inner{\nabla f(v)}{w-v} \ = \ f(u)+\inner{\nabla f(v)}{w-u} ,
    \end{equation*}
    the last equality because $f(v) = f(u)$ and $\inner{\nabla f(v)}{u-v} = 0$.
    This is the subgradient inequality for $\nabla f(v)$ at $u$, so $\nabla f(v) \in \partial f(u) = \curlybr{\nabla f(u)}$, the last set equality because $f$ is differentiable \citep[Theorem~25.1]{rockafellar1997convex}.
\end{proof}

\begin{lemma}
\label{lem:generic-decomposition}
    For any three points $x_{+}, x, z \in \R^d$,
    \begin{equation*}
        \norm{x_{+} - z}^2 = \norm{ x - z}^2 - 2\inner{x - x_{+}}{x_{+} - z} -\norm{x - x_{+}}^2.
    \end{equation*}
\end{lemma}

\begin{proof}
    We write $x-z = (x-x_{+})+(x_{+}-z)$. 
    Hence,
    \begin{equation*}
        \norm{x-z}^2 = \norm{(x-x_{+}) + (x_{+}-z)}^2 = \norm{x-x_{+}}^2 + \norm{x_{+}-z}^2 + 2\inner{x-x_{+}}{x_{+}-z}.
    \end{equation*}
    Rearranging this identity yields
    \begin{equation*}
        \norm{x_{+}-z}^2 = \norm{x-z}^2 - 2\inner{x - x_{+}}{x_{+} - z} -\norm{x - x_{+}}^2.
    \end{equation*}
\end{proof}

\begin{lemma}
\label{lem:normal_cone_ball}
    Let $x \in \R^d$ and $t > 0$, and define the closed Euclidean ball
    \begin{equation*}
        \cB(x,t) \eqdef \curlybr{z\in \R^d \st \norm{z-x}\le t}.
    \end{equation*}
    Then, for any $y \in \R^d$, the normal cone of $\cB(x,t)$ at $y$ is given by
    \begin{equation*}
        N_{\cB(x,t)}(y) =
        \begin{cases}
        \emptyset, & y \notin \cB(x,t),\\[1ex]
        \curlybr{0}, & \norm{y-x} < t,\\[1ex]
        \curlybr{\lambda (y-x)\st \lambda\ge 0}, & \norm{y-x} = t.
        \end{cases}
    \end{equation*}
\end{lemma}

\begin{proof}
    Recall that, for a convex set $\cC$ and $y \in \cC$, $N_\cC(y) = \curlybr{v \in \R^d \st \inner{v}{z-y} \le 0 \ \forall z \in \cC}$, while $N_\cC(y) = \emptyset$ for $y \notin \cC$.
    This settles the case $y \notin \cB(x,t)$. 

    For the remainder of this proof, we write $\cB \eqdef \cB(x,t)$.

    \emph{Case $\norm{y-x} < t$.}
    Clearly $0 \in N_\cB(y)$. 
    Let $v \in N_\cB(y)$ and suppose $v \neq 0$.
    Set $\varepsilon \eqdef t - \norm{y-x} > 0$ and $z \eqdef y + \varepsilon v/\norm{v}$.
    Then $\norm{z-x} \le \norm{y-x} + \varepsilon = t$, so $z \in \cB$, but $\inner{v}{z-y} = \varepsilon\norm{v} > 0$, a contradiction.
    Hence $N_\cB(y) = \curlybr{0}$.

    \emph{Case $\norm{y-x} = t$.}
    Let $\lambda \ge 0$ and $z \in \cB$. 
    By Cauchy--Schwarz,
    \begin{align*}
        \inner{\lambda(y-x)}{z-y} & = \lambda\rbr{\inner{y-x}{z-x} - \norm{y-x}^2} \\
        & \le \lambda\rbr{\norm{y-x}\norm{z-x} - t^2} \\
        & \le \lambda\rbr{t^2 - t^2} = 0 ,        
    \end{align*}
    so $\lambda(y-x) \in N_\cB(y)$.
    
    Next, let $v \in N_\cB(y)$, set $u \eqdef (y-x)/t$ (a unit vector), and decompose $v = \mu u + w$ with $\mu \eqdef \inner{v}{u}$ and $\inner{w}{u} = 0$.
    Choosing $z = x \in \cB$ gives $0 \ge \inner{v}{x-y} = -t\mu$, so $\mu \ge 0$.
    Suppose $w \neq 0$, let $\hat w \eqdef w/\norm{w}$, and for $\beta \in (0,\pi/2]$ consider $z_\beta \eqdef x + t\rbr{\cos\beta\, u + \sin\beta\, \hat w}$.
    Since $u \perp \hat w$ are unit vectors, $\norm{z_\beta - x} = t$, so $z_\beta \in \cB$.
    Using $z_\beta - y = t\rbr{(\cos\beta - 1)u + \sin\beta\, \hat w}$, $1-\cos\beta \le \beta^2/2$ and $\sin\beta \ge 2\beta/\pi$,
    \begin{equation*}
        \inner{v}{z_\beta - y} = t\rbr{\mu(\cos\beta-1) + \norm{w}\sin\beta} \ge t\beta\rbr{\tfrac{2}{\pi}\norm{w} - \tfrac{\mu}{2}\beta} \ ,
    \end{equation*}
    which is positive for all sufficiently small $\beta > 0$, contradicting $v \in N_\cB(y)$.
    Therefore $w = 0$ and $v = \mu u = (\mu/t)(y-x)$ with $\mu/t \ge 0$.
\end{proof}

\subsection{Localization and Fejér monotonicity}
\label{sub:localization_and_fejer_monotonicity}

The following two results underlie parts \ref{it:lmoproj-localizes} and \ref{it:lmoproj-distance} of \Cref{thm:lmo_is_projection}, respectively.
\Cref{thm:projection_fejer_halfspace} further underlies the Fej\'er arguments of \Cref{app:band,app:localrate}.

\begin{theorem}
\label{thm:alpha_star_exact}
    Fix a point $x \in \R^d$.
    \begin{enumerate}
        \item
        The quantity
        \begin{equation}
            \label{eq:alpha_star_only_def}
            \alpha_\star(x) \eqdef \sup_{x_\star\in \cX_\star}\inner{\nabla f(x)}{x_\star}
        \end{equation}
        is finite.
        In particular,
        \begin{equation}
            \label{eq:alpha_star_le_alpha_cvx_1}
            \alpha_\star(x) \leq \alpha_{\rm cvx}(x),
        \end{equation}
        where
        \begin{equation}
            \label{eq:alpha_cvx_def}
            \alpha_{\rm cvx}(x) \eqdef \inner{\nabla f(x)}{x} - \rbr{f(x)-f_\star} < +\infty.
        \end{equation}
        \item
        For $\alpha \in \R$, the following statements are equivalent:
        \begin{enumerate}
            \item
            $\cX_\star \subseteq \cH(\nabla f(x),\alpha)$,

            \item
            $\inner{\nabla f(x)}{x_\star} \le \alpha$ for every $x_\star \in \cX_\star$,
            
            \item
            $\alpha_\star(x) \le \alpha$.
        \end{enumerate}
        \item
        The following inclusions hold:
        \begin{equation}
            \label{eq:alpha_star}
            \cX_\star \subseteq \cH(\nabla f(x),\alpha_\star(x)) \subseteq \cH(\nabla f(x),\alpha_{\rm cvx}(x)).
        \end{equation}
    \end{enumerate}
\end{theorem}

\begin{proof}
    \begin{enumerate}
        \item
        Let $x_\star \in \cX_\star$ be arbitrary.
        Since $f$ is convex and differentiable on $\R^d$ (\Cref{ass:main}), the first-order convexity inequality gives
        \begin{equation*}
            f_\star = f(x_\star) \ge f(x)+\inner{\nabla f(x)}{x_\star-x} .
        \end{equation*}
        Rearranging, we obtain
        \begin{equation*}
            \inner{\nabla f(x)}{x_\star} \le \inner{\nabla f(x)}{x} - \rbr{f(x)-f_\star} = \alpha_{\rm cvx}(x) < + \infty.
        \end{equation*}
        Since this inequality holds for every $x_\star \in \cX_\star$, we may take the supremum over $x_\star \in \cX_\star$ and get
        \begin{equation*}
            \alpha_\star(x) = \sup_{x_\star\in \cX_\star}\inner{\nabla f(x)}{x_\star} \le \alpha_{\rm cvx}(x) < +\infty.
        \end{equation*}
        Moreover, since $\cX_\star \neq \emptyset$ (\Cref{ass:main}) the supremum is over a non-empty set, hence $\alpha_\star(x) >-\infty$.
        This proves \eqref{eq:alpha_star_le_alpha_cvx_1}.
        \item
        Let $g = \nabla f(x)$.
        By definition of $\cH(g,\alpha)$, for any $x_\star \in \cX_\star$,
        \begin{equation*}
            x_\star \in \cH(g,\alpha) \qquad\Longleftrightarrow\qquad \inner{g}{x_\star} \le \alpha.
        \end{equation*}
        Therefore,
        \begin{equation*}
            \cX_\star \subseteq \cH(g,\alpha) \qquad\Longleftrightarrow\qquad \inner{g}{x_\star} \le \alpha \quad \forall x_\star \in \cX_\star.
        \end{equation*}
        This proves the equivalence of i. and ii.

        Next, statement ii. is equivalent to saying that $\alpha$ is an upper bound on the set
        \begin{equation*}
            \curlybr{\inner{g}{x_\star}\st x_\star\in \cX_\star}.
        \end{equation*}
        By the definition \eqref{eq:alpha_star_only_def} of $\alpha_\star$, this is equivalent to
        \begin{equation*}
            \sup_{x_\star\in \cX_\star}\inner{g}{x_\star} \le \alpha,
        \end{equation*}
        that is, $\alpha_\star(x) \le \alpha$. 
        This proves the equivalence of ii. and iii.
        \item
        This follows immediately by combining part~2 and \eqref{eq:alpha_star_le_alpha_cvx_1}, and using the definition \eqref{eq:halfspace_level}.
    \end{enumerate}
\end{proof}

\begin{theorem}
\label{thm:projection_fejer_halfspace}
    Fix $x_0 \in \cX$, and define
    \begin{equation*}
        x_{k+1} \eqdef \Proj_{\cX \cap \cH(\nabla f(x_k),\alpha_k)}(x_k), \qquad k = 0,1,\dots,
    \end{equation*}
    where $\alpha_k \geq \alpha_\star(x_k)$.
    Then, for all $k \geq 0$, we have
    \begin{equation}
        \label{eq:sq-distance-shrink}
        \norm{x_{k+1} - x_{\star}}^2 \leq \norm{x_k - x_{\star}}^2 - \norm{x_k - x_{k+1}}^2 \qquad \forall x_\star \in \cX_\star,
    \end{equation}
    and therefore, for all $K \geq 1$, we have
    \begin{equation}
        \label{eq:min-sq-distance-bound}
        \min_{0\leq k \leq K-1} \norm{x_k - x_{k+1}}^2 \leq \frac{\dist^2(x_0,\cX_\star)}{K}.
    \end{equation}
\end{theorem}

\begin{proof}
    By \Cref{lem:generic-decomposition}, we know that
    \begin{equation*}
        \norm{x_{k+1} - x_\star}^2 = \norm{ x_k - x_\star}^2 - 2\inner{x_k - x_{k+1}}{x_{k+1} - x_\star} - \norm{x_k - x_{k+1}}^2
    \end{equation*}
    for all $x_\star \in \cX_\star$.
    Hence, we only need to show that
    \begin{equation}
        \label{eqL:i-fd8798fgfg}
        \inner{x_k - x_{k+1}}{x_{k+1} - x_\star} \ge 0.
    \end{equation}

    Let $g_k = \nabla f(x_k)$.
    Since $\alpha_k \geq \alpha_\star(x_k)$, \Cref{thm:alpha_star_exact} implies that $\cX_\star \subseteq \cH(g_k,\alpha_k)$.
    Since also $\cX_\star \subseteq \cX$, we have
    \begin{equation*}
        x_\star \in \cX \cap \cH(g_k,\alpha_k), \qquad \forall x_\star \in \cX_\star.
    \end{equation*}
    Because $x_{k+1}$ is the Euclidean projection of $x_k$ onto the non-empty closed convex set $\cX \cap \cH(g_k,\alpha_k)$, the standard projection optimality condition yields
    \begin{equation*}
        \inner{x_k - x_{k+1}}{y - x_{k+1}} \le 0, \qquad \forall y \in \cX \cap \cH(g_k,\alpha_k).
    \end{equation*}
    Since $\cX_\star \subseteq \cX \cap \cH(g_k,\alpha_k)$, the above inequality holds for all $y = x_\star \in \cX_\star$, establishing \eqref{eqL:i-fd8798fgfg}.

    Finally, \eqref{eq:min-sq-distance-bound} follows by summing up inequalities \eqref{eq:sq-distance-shrink} for $k = 0,\dots,K-1$.
    Indeed, the inequalities
    \begin{align*}
        \norm{x_1-x_\star}^2 & \leq \norm{x_0-x_\star}^2 - \norm{x_0-x_1}^2\\
        \norm{x_2-x_\star}^2 & \leq \norm{x_1-x_\star}^2 - \norm{x_1-x_2}^2\\
        & \vdots \\
        \norm{x_K-x_\star}^2 & \leq \norm{x_{K-1}-x_\star}^2 - \norm{x_{K-1}-x_K}^2
    \end{align*}
    add up to
    \begin{equation*}
        \norm{x_K-x_\star}^2 \leq \norm{x_0-x_\star}^2 - \sum_{k=0}^{K-1}\norm{x_{k} - x_{k+1}}^2,
    \end{equation*}
    and after rearranging and dividing by $K$, we get
    \begin{equation*}
        \min_{0\leq k \leq K-1} \norm{x_k - x_{k+1}}^2 \leq \frac{1}{K}\sum_{k=0}^{K-1}\norm{x_{k} - x_{k+1}}^2 \leq \frac{\norm{x_0-x_\star}^2 - \norm{x_K-x_\star}^2}{K} \leq \frac{\norm{x_0-x_\star}^2 }{K}.
    \end{equation*}
    It only remains to choose
    \begin{equation*}
        x_\star = \argmin_{z \in \cX_\star} \norm{x_0 - z},
    \end{equation*}
    so that $\norm{x_0-x_\star} = \dist(x_0,\cX_\star)$.
\end{proof}

\subsection{Three thresholds}
\label{sub:three_thresholds}

For $x \in \R^d$ and $t \in \R$, we define
\begin{equation}
    \label{eq:alpha-proj-def}
    \alpha_{\rm proj}(x,t) \eqdef \inner{\nabla f(x)}{x} - t\norm{\nabla f(x)} .
\end{equation}

\begin{lemma}
\label{lem:alpha_t_vs_alpha_cvx}
    Fix $x \in \R^d$ and assume $\nabla f(x) \neq 0$.
    \begin{enumerate}
        \item
        If $t \le t(x)$, then
        \begin{equation*}
            \alpha_{\rm cvx}(x) \le \alpha_{\rm proj}(x,t).
        \end{equation*}
        
        \item
        If $t = t(x)$, then
        \begin{equation*}
            \alpha_{\rm cvx}(x) = \alpha_{\rm proj}(x,t).
        \end{equation*}
    \end{enumerate}
    Consequently, if $t \le t(x)$, then
    \begin{equation}
        \label{eq:alpha_t_cvx_inclusion}
        \cX_\star \subseteq \cH(\nabla f(x),\alpha_{\rm cvx}(x)) \subseteq \cH(\nabla f(x),\alpha_{\rm proj}(x,t)).
    \end{equation}
\end{lemma}

\begin{proof}
    By definition, $\alpha_{\rm cvx}(x) = \inner{\nabla f(x)}{x}-\rbr{f(x)-f_\star}$, while $\alpha_{\rm proj}(x,t) = \inner{\nabla f(x)}{x}-t\norm{\nabla f(x)}$. 
    Hence
    \begin{equation*}
        \alpha_{\rm proj}(x,t)-\alpha_{\rm cvx}(x) = \rbr{f(x)-f_\star}-t\norm{\nabla f(x)}.
    \end{equation*}
    Using the definition \eqref{def:polyak-radius} of $t(x)$, we can rewrite this as
    \begin{equation*}
        \alpha_{\rm proj}(x,t)-\alpha_{\rm cvx}(x) = \norm{\nabla f(x)}\rbr{t(x)-t}.
    \end{equation*}
    Therefore, if $t \le t(x)$, then $\alpha_{\rm proj}(x,t)-\alpha_{\rm cvx}(x) \ge 0,$ which proves $\alpha_{\rm cvx}(x) \le \alpha_{\rm proj}(x,t)$. 
    If $t = t(x)$, then $\alpha_{\rm proj}(x,t)-\alpha_{\rm cvx}(x) = 0$, and thus $\alpha_{\rm cvx}(x) = \alpha_{\rm proj}(x,t).$
    The first inclusion in \eqref{eq:alpha_t_cvx_inclusion} is \eqref{eq:alpha_star_le_alpha_cvx_1} of \Cref{thm:alpha_star_exact} together with part~2 of that theorem; the second is $\alpha_{\rm cvx}(x) \le \alpha_{\rm proj}(x,t)$ and the monotonicity of $\cH(\nabla f(x),\cdot)$.
\end{proof}

\begin{lemma}
\label{lem:alpha_t_le_alpha_lmo}
    Fix $x \in \cX \setminus \cX_\star$.
    For $t > 0$, define
    \begin{equation}
        \label{eq:alpha-lmo-def}
        \alpha_{\rm lmo}(x,t) \eqdef \min_{z\in \cX \cap \cB(x,t)} \inner{\nabla f(x)}{z}.
    \end{equation}    
    Then
    \begin{equation}
        \label{eq:alpha_t_lmo}
        \alpha_{\rm proj}(x,t) \le \alpha_{\rm lmo}(x,t).
    \end{equation}
    Consequently, if $0 < t \le t(x)$, then
    \begin{equation}
        \label{eq:alpha_t_lmo_inclusion}
        \cX_\star \subseteq \cH(\nabla f(x),\alpha_{\rm proj}(x,t)) \subseteq \cH(\nabla f(x),\alpha_{\rm lmo}(x,t)).
    \end{equation}
\end{lemma}

\begin{proof}
    Let $g = \nabla f(x)$, $\alpha_{\rm lmo} = \alpha_{\rm lmo}(x,t)$ and $\alpha_{\rm proj} = \alpha_{\rm proj}(x,t)$.

    \paragraph{Step 1: Non-zero gradient.}
    Since $x \notin \cX_\star$, we have $g \neq 0$ (\Cref{lem:zero_grad}).

    \paragraph{Step 2: Inequality between the thresholds.}
    Fix any $z \in \cX \cap \cB(x,t)$.
    Since $z \in \cB(x,t)$, we have $\norm{z-x} \le t$. 
    Therefore, by the Cauchy--Schwarz inequality,
    \begin{equation*}
        \inner{g}{z} = \inner{g}{x}+\inner{g}{z-x} \ge \inner{g}{x}-\norm{g}\,\norm{z-x} \ge \inner{g}{x}-t\norm{g} = \alpha_{\rm proj}.
    \end{equation*}
    Since this holds for every $z \in \cX \cap \cB(x,t)$, taking the minimum over $z \in \cX \cap \cB(x,t)$ yields
    \begin{equation*}
        \alpha_{\rm lmo} = \min_{z\in \cX \cap \cB(x,t)} \inner{g}{z} \ge \alpha_{\rm proj}.
    \end{equation*}
    This proves \eqref{eq:alpha_t_lmo}.

    \paragraph{Step 3: Relationship between the half-spaces.}
    Now assume in addition that $0 < t \le t(x)$.
    By \eqref{eq:alpha_t_cvx_inclusion} of \Cref{lem:alpha_t_vs_alpha_cvx}, we have
    \begin{equation*}
        \cX_\star \subseteq \cH(g,\alpha_{\rm proj}).
    \end{equation*}
    Moreover, the inclusion
    \begin{equation*}
        \cH(g,\alpha_{\rm proj}) \subseteq \cH(g,\alpha_{\rm lmo})
    \end{equation*}
    follows from \eqref{eq:alpha_t_lmo}.
    Combining the last two inclusions establishes \eqref{eq:alpha_t_lmo_inclusion}.
\end{proof}

\subsection{Identification, step length, and the proof of \Cref{thm:lmo_is_projection}}

\begin{theorem}
\label{thm:lmo_boundary_and_projection}
    Fix $x \in \cX \setminus \cX_\star$, choose
    \begin{equation}
        \label{eq:polyak_radius_bound_combined}
        0 < t \le t(x) \eqdef \frac{f(x)-f_\star}{\norm{\nabla f(x)}},
    \end{equation}
    and let
    \begin{equation}
        \label{eq:lmo_step_combined}
        x_{\rm lmo} \in \argmin_{z\in \cX \cap \cB(x,t)} \inner{\nabla f(x)}{z},
    \end{equation}
    with
    \begin{equation}
        \label{eq:alpha_lmo_combined}
        \alpha_{\rm lmo}(x,t) \eqdef \min_{z\in \cX \cap \cB(x,t)} \inner{\nabla f(x)}{z} = \inner{\nabla f(x)}{x_{\rm lmo}} .
    \end{equation}
    Then,
    \begin{enumerate}
        \item \label{it:bdry-sphere}
        $x_{\rm lmo}$ lies on the boundary of the ball $\cB(x,t)$, that is,
        \begin{equation}
            \label{eq:lmo_step_has_length_t_combined}
            \norm{x_{\rm lmo}-x} = t;
        \end{equation}

        \item \label{it:bdry-proj}
        $x_{\rm lmo}$ is the Euclidean projection of $x$ onto the intersection of $\cX$ with the half-space
        \begin{equation*}
            \cH_{\rm lmo} \eqdef \cH(\nabla f(x),\alpha_{\rm lmo}(x,t)),
        \end{equation*}
        that is,
        \begin{equation}
            \label{eq:lmo_equals_projection_halfspace_combined}
            x_{\rm lmo} = \Proj_{\cX \cap \cH_{\rm lmo}}(x).
        \end{equation}
    \end{enumerate}
\end{theorem}

\begin{proof}
    Let $g \eqdef \nabla f(x)$.

    \paragraph{Step 1: Non-zero gradient.}
    Since $x \notin \cX_\star$, we have $\nabla f(x) \neq 0$ (\Cref{lem:zero_grad}).

    \paragraph{Step 2: Non-zero step.}
    We now show that
    \begin{equation}
        \label{eq:non_zero_step}
        x_{\rm lmo} \neq x,
    \end{equation}
    which will be useful later.

    Assume, for the sake of contradiction, that $x_{\rm lmo} = x$. 
    Since $x_{\rm lmo}$ solves \eqref{eq:lmo_step_combined}, this implies that $x$ minimizes the linear functional
    \begin{equation*}
        z \mapsto \inner{\nabla f(x)}{z}
    \end{equation*}
    over $\cX \cap \cB(x,t)$.
    Hence,
    \begin{equation}
        \label{eq:optimality_x_itself_lemma}
        \inner{\nabla f(x)}{z-x} \ge 0, \qquad \forall z \in \cX \cap \cB(x,t).
    \end{equation}

    Now fix any $x_\star \in \cX_\star$.
    Since $f$ is convex and differentiable (\Cref{ass:main}), we have
    \begin{equation}
        \label{eq:cvx_0978ttfyif}
        f_\star = f(x_\star) \ge f(x)+\inner{\nabla f(x)}{x_\star-x}.
    \end{equation}
    Since $x \notin \cX_\star$, we have $f(x) > f_\star$, and therefore
    \begin{equation}
        \label{eq:strict_inner_positive_lemma}
        \inner{\nabla f(x)}{x-x_\star} \overset{\eqref{eq:cvx_0978ttfyif}}{\geq} f(x) - f_\star > 0.
    \end{equation}

    Because the radius $t$ satisfies \eqref{eq:polyak_radius_bound_combined}, and
    \begin{equation*}
        f(x)-f_\star \overset{\eqref{eq:cvx_0978ttfyif}}{\leq} \inner{\nabla f(x)}{x-x_\star} \le \norm{\nabla f(x)}\,\norm{x-x_\star},
    \end{equation*}
    it follows that $0 < t \le \norm{x-x_\star}$. 
    The scalar $\lambda \eqdef \frac{t}{\norm{x-x_\star}}$ thus belongs to $(0,1]$.
    Define
    \begin{equation}
        \label{eq:0970fyd808fg}
        z_t \eqdef (1-\lambda)x+\lambda x_\star = x-\frac{t}{\norm{x-x_\star}}(x-x_\star).
    \end{equation}
    Since $x,x_\star \in \cX$ and $\cX$ is convex (\Cref{ass:main}), we have $z_t \in \cX$ also. 
    Therefore,
    \begin{equation*}
        \norm{z_t-x} \overset{\eqref{eq:0970fyd808fg}}{=} \frac{t}{\norm{x-x_\star}}\norm{x-x_\star} = t.
    \end{equation*}
    Hence, $z_t \in \cB(x,t)$, which implies that $z_t \in \cX \cap \cB(x,t)$. 
    Applying \eqref{eq:optimality_x_itself_lemma} with $z = z_t$, we obtain
    \begin{equation*}
        \inner{\nabla f(x)}{z_t-x} \ge 0.
    \end{equation*}
    On the other hand,
    \begin{equation*}
        \inner{\nabla f(x)}{z_t-x} \overset{\eqref{eq:0970fyd808fg}}{=} -\frac{t}{\norm{x-x_\star}}\inner{\nabla f(x)}{x-x_\star}.
    \end{equation*}
    By \eqref{eq:strict_inner_positive_lemma}, the right-hand side is strictly negative, and hence
    \begin{equation*}
        \inner{\nabla f(x)}{z_t-x} < 0,
    \end{equation*}
    which contradicts inequality \eqref{eq:optimality_x_itself_lemma}.
    We therefore conclude that $x_{\rm lmo} \neq x$.

    \paragraph{Step 3: An upper bound on the radius.}
    Fix any $x_\star \in \cX_\star$.
    Since $f$ is convex and differentiable (\Cref{ass:main}),
    \begin{equation}
        \label{eq:h8fd+=-9=08fg}
        f_\star = f(x_\star) \ge f(x)+\inner{\nabla f(x)}{x_\star-x}.
    \end{equation}
    Combining this with \eqref{eq:polyak_radius_bound_combined}, we get
    \begin{equation}
        \label{eq:combined_admissibility_corrected}
        0 < t \overset{\eqref{eq:polyak_radius_bound_combined}}{\leq} t(x) \overset{\eqref{def:polyak-radius}}{=} \frac{f(x)-f_\star}{\norm{\nabla f(x)}} \overset{\eqref{eq:h8fd+=-9=08fg}}{\leq} \frac{\inner{\nabla f(x)}{x-x_\star}}{\norm{\nabla f(x)}}.
    \end{equation}

    \paragraph{Step 4: Optimality conditions for the local linear minimization problem.}
    Since $x_{\rm lmo}$ solves the optimization problem
    \begin{equation*}
        \min_{z\in \cX \cap \cB(x,t)} \inner{\nabla f(x)}{z},
    \end{equation*}
    the first-order optimality condition \citep[Theorem~27.4]{rockafellar1997convex} gives $0 \in \nabla f(x)+N_{\cX \cap\cB(x,t)}(x_{\rm lmo})$, where $N_\cC(z)$ denotes the normal cone of the set $\cC$ at point $z$.
    The sum rule $N_{\cX \cap\cB(x,t)} = N_{\cX}+N_{\cB(x,t)}$ applies here as both sets are convex and their relative interiors meet \citep[Corollary~23.8.1]{rockafellar1997convex}.
    Indeed, $\ri \cB(x,t) = \interior \cB(x,t) = \curlybr{z \st \norm{z-x} < t}$ is an open set containing $x$.
    Since $\cX$ is non-empty and convex, $x \in \cX \subseteq \cl \cX = \cl\rbr{\ri \cX}$ \citep[Theorem~6.3]{rockafellar1997convex}, so this open neighborhood of $x$ contains a point of $\ri\cX$, i.e. $\ri\cX \cap \ri\cB(x,t) \neq \emptyset$.

    \Cref{lem:normal_cone_ball} says that
    \begin{equation}
        \label{eq:normal_cone_ball_xx}
        N_{\cB(x,t)}(y) =
        \begin{cases}
        \emptyset, & y \notin \cB(x,t),\\[1ex]
        \curlybr{0}, & \norm{y-x} < t,\\[1ex]
        \curlybr{\lambda' (y-x)\st \lambda' \geq 0}, & \norm{y-x} = t.
        \end{cases}
    \end{equation}
    Since $x_{\rm lmo} \in \cB(x,t)$, we conclude that
    \begin{equation}
        \label{eq:-8-y8f9dufd}
        N_{\cB(x,t)}(x_{\rm lmo}) = \begin{cases} \curlybr{0}, & \norm{x_{\rm lmo}-x} < t,\\[1ex]
        \curlybr{\lambda' (x_{\rm lmo}-x)\st \lambda' \geq 0}, & \norm{x_{\rm lmo}-x} = t.
        \end{cases}
    \end{equation}
    Hence, there exist $n \in N_{\cX}(x_{\rm lmo})$ and $\lambda \ge 0$ such that
    \begin{equation}
        \label{eq:kkt_corrected}
        \nabla f(x) + n + \lambda(x_{\rm lmo}-x) = 0.
    \end{equation}

    \paragraph{Step 5: Moving to the boundary of the ball.}
    Claim \ref{it:bdry-sphere} has been established by \citet[Theorem~3.1(ii)]{richtarik2026local}.
    We include the proof as we re-use one of the intermediate results below.

    We already know that $x_{\rm lmo} \in \cB(x,t)$, and therefore
    \begin{equation}
        \label{eq:norm_upper_corrected}
        \norm{x_{\rm lmo}-x} \le t.
    \end{equation}
    It remains to prove the reverse inequality.
    We distinguish two cases.

    \paragraph{$\bullet$ Case 1: $\lambda>0$.}
    By the construction of \eqref{eq:kkt_corrected}, $\lambda (x_{\rm lmo}-x) \in N_{\cB(x,t)}(x_{\rm lmo})$.
    Since $x_{\rm lmo} \neq x$ and $\lambda > 0$, characterization \eqref{eq:normal_cone_ball_xx} of the normal cone $N_{\cB(x,t)}(x_{\rm lmo})$ implies that
    \begin{equation*}
        \norm{x_{\rm lmo}-x} = t.
    \end{equation*}

    \paragraph{$\bullet$ Case 2: $\lambda=0$.}
    In this case, \eqref{eq:kkt_corrected} becomes $\nabla f(x) + n = 0,$ i.e.,
    \begin{equation*}
        -\nabla f(x) \in N_{\cX}(x_{\rm lmo}).
    \end{equation*}
    By the definition of the normal cone, this implies $\inner{-\nabla f(x)}{z-x_{\rm lmo}} \le 0$ for all $z \in \cX$. 
    Equivalently,
    \begin{equation*}
        \inner{\nabla f(x)}{x_{\rm lmo}-z} \leq 0, \qquad \forall z \in \cX.
    \end{equation*}
    In particular, for any $x_\star \in \cX_\star \subseteq \cX$,
    \begin{equation}
        \label{eq:normalcone_xstar_equiv}
        \inner{\nabla f(x)}{x_{\rm lmo}-x_\star} \le 0.
    \end{equation}

    Using the decomposition
    \begin{equation*}
        x_{\rm lmo}-x_\star = (x-x_\star)+(x_{\rm lmo}-x),
    \end{equation*}
    we obtain
    \begin{equation}
        \label{eq:inner_decomp}
        \inner{\nabla f(x)}{x_{\rm lmo}-x_\star} = \inner{\nabla f(x)}{x-x_\star} + \inner{\nabla f(x)}{x_{\rm lmo}-x}.
    \end{equation}
    By the Cauchy--Schwarz inequality,
    \begin{equation}
        \label{eq:key_three_cs}
        \inner{\nabla f(x)}{x_{\rm lmo}-x} \ge -\norm{\nabla f(x)} \norm{x_{\rm lmo}-x} \overset{\eqref{eq:norm_upper_corrected}}{\ge} -\norm{\nabla f(x)} t.
    \end{equation}
    By combining this with \eqref{eq:inner_decomp}, we get
    \begin{equation*}
        \inner{\nabla f(x)}{x_{\rm lmo}-x_\star} \ge \inner{\nabla f(x)}{x-x_\star}-t\norm{ \nabla f(x)}.
    \end{equation*}
    Using \eqref{eq:combined_admissibility_corrected}, we conclude that
    \begin{equation}
        \label{eq:halfspace_lower_corrected}
        \inner{\nabla f(x)}{x_{\rm lmo}-x_\star} \ge 0.
    \end{equation}
    Combining \eqref{eq:normalcone_xstar_equiv} and \eqref{eq:halfspace_lower_corrected}, we get
    \begin{equation*}
        \inner{\nabla f(x)}{x_{\rm lmo}-x_\star} = 0.
    \end{equation*}
    Hence equality must hold in the preceding chain of inequalities, and in particular in \eqref{eq:key_three_cs}.
    This implies
    \begin{equation*}
        \norm{x_{\rm lmo}-x} = t ,
    \end{equation*}
    which proves \eqref{eq:lmo_step_has_length_t_combined}.

\paragraph{Step 6: Projection characterization.}
    We now prove that
    \begin{equation*}
        x_{\rm lmo} = \Proj_{\cX \cap \cH(\nabla f(x),\alpha_{\rm lmo}(x,t))}(x).
    \end{equation*}
    Let $\cC \subseteq \R^d$ be a non-empty closed convex set and let $x \in \R^d$. A point $u \in \cC$ satisfies
    \begin{equation*}
        u = \Proj_{\cC}(x)
    \end{equation*}
    if and only if
    \begin{equation}
        \label{eq:proj_characterization_corrected}
        \inner{x-u}{y-u} \le 0, \qquad \forall y \in \cC.
    \end{equation}
    Therefore, it suffices to prove that
    \begin{equation}
        \label{eq:goal_projection_corrected}
        \inner{x-x_{\rm lmo}}{y-x_{\rm lmo}} \le 0, \qquad \forall y \in \cX \cap \cH(\nabla f(x),\alpha_{\rm lmo}(x,t)).
    \end{equation}

    We again distinguish the same two cases as above.

    \paragraph{$\bullet$ Case 1: $\lambda>0$.}
    Fix any $y \in \cX \cap \cH(\nabla f(x),\alpha_{\rm lmo}(x,t))$.
    Since $n \in N_{\cX}(x_{\rm lmo})$ and $y \in \cX$,
    \begin{equation*}
        \inner{n}{y-x_{\rm lmo}} \le 0.
    \end{equation*}
    Also, since $y \in \cH(\nabla f(x),\alpha_{\rm lmo}(x,t))$ and $\alpha_{\rm lmo}(x,t) = \inner{\nabla f(x)}{x_{\rm lmo}}$,
    \begin{equation*}
        \inner{\nabla f(x)}{y-x_{\rm lmo}} \le 0.
    \end{equation*}
    Taking the inner product of \eqref{eq:kkt_corrected} with $y-x_{\rm lmo}$, we obtain
    \begin{equation*}
        \inner{\nabla f(x)}{y-x_{\rm lmo}} + \inner{n}{y-x_{\rm lmo}} + \lambda\inner{x_{\rm lmo}-x}{y-x_{\rm lmo}} = 0.
    \end{equation*}
    The first two terms are non-positive, and hence $\lambda\inner{x_{\rm lmo}-x}{y-x_{\rm lmo}} \ge 0$. 
    Since $\lambda > 0$,
    \begin{equation*}
        \inner{x-x_{\rm lmo}}{y-x_{\rm lmo}} \le 0.
    \end{equation*}
    Thus \eqref{eq:goal_projection_corrected} holds in this case.

    \paragraph{$\bullet$ Case 2: $\lambda=0$.}
    As proved above, equality must hold in the Cauchy--Schwarz inequality \eqref{eq:key_three_cs}, and therefore
    \begin{equation*}
        x-x_{\rm lmo} = t\,\frac{\nabla f(x)}{\norm{\nabla f(x)}}.
    \end{equation*}
    Now fix any $y \in \cX \cap \cH(\nabla f(x),\alpha_{\rm lmo}(x,t))$.
    Since $y \in \cH(\nabla f(x),\alpha_{\rm lmo}(x,t))$,
    \begin{equation*}
        \inner{\nabla f(x)}{y-x_{\rm lmo}} \le 0.
    \end{equation*}
    Therefore
    \begin{equation*}
        \inner{x-x_{\rm lmo}}{y-x_{\rm lmo}} = \frac{t}{\norm{\nabla f(x)}}\inner{\nabla f(x)}{y-x_{\rm lmo}} \le 0.
    \end{equation*}
    Thus \eqref{eq:goal_projection_corrected} holds in this case as well.
\end{proof}

We now have all the parts to establish \Cref{thm:lmo_is_projection}, which we restate here for convenience.

\lmoprojthm*

The theorem is phrased in depth terms; 
we argue in levels and translate with \eqref{eq:dictionary} and \eqref{eq:level_depth_table}, under which $\cH_x(\nabla f(x),s_{\rm lmo}(x,t)) = \cH(\nabla f(x),\alpha_{\rm lmo}(x,t))$.

\begin{proof}
    Write $g \eqdef \nabla f(x)$, which is non-zero by \Cref{lem:zero_grad}.
    \begin{enumerate}
        \item 
        \Cref{thm:alpha_star_exact,lem:alpha_t_vs_alpha_cvx,lem:alpha_t_le_alpha_lmo} show that $\alpha_\star(x) \le \alpha_{\rm cvx}(x) \le \alpha_{\rm proj}(x,t) \le \alpha_{\rm lmo}(x,t)$, which by \eqref{eq:level_depth_table} is the chain $s_{\rm lmo}(x,t) \le t \le t(x) \le s_\star(x)$.
        For the remaining strict inequality $s_{\rm lmo}(x,t) > 0$, fix any $x_\star \in \cX_\star$ and set $\mu \eqdef t/\norm{x-x_\star}$, which lies in $(0,1]$ because $0 < t \le t(x) \le \inner{g}{x-x_\star}/\norm{g} \le \norm{x-x_\star}$ by \eqref{eq:combined_admissibility_corrected} and Cauchy--Schwarz.
        The point $z \eqdef (1-\mu)x+\mu x_\star$ lies in $\cX\cap\cB(x,t)$, so $\alpha_{\rm lmo}(x,t) \le \inner{g}{z} = \inner{g}{x}-\mu\inner{g}{x-x_\star}$ and therefore $s_{\rm lmo}(x,t) = \rbr{\inner{g}{x}-\alpha_{\rm lmo}(x,t)}/\norm{g} \ge \mu\inner{g}{x-x_\star}/\norm{g} \ge \mu\rbr{f(x)-f_\star}/\norm{g} > 0$; equivalently, $x \notin \cH_x(\nabla f(x),s_{\rm lmo}(x,t))$.
        The same computation is recorded, in the notation of the local Frank--Wolfe gap, as \Cref{lem:local_fw_gap}.
        Since $\cX_\star \subseteq \cH(\nabla f(x), \alpha_\star(x))$ (\Cref{thm:alpha_star_exact}) and the half-spaces are increasing in the level $\alpha$, we have
        \begin{equation*}
            \cX_\star \subseteq \cH(\nabla f(x), \alpha_\star(x)) \subseteq \cH(\nabla f(x), \alpha_{\rm lmo}(x,t)) .
        \end{equation*}

        \item 
        This is the content of \eqref{eq:lmo_equals_projection_halfspace_combined}, part~\ref{it:bdry-proj} of \Cref{thm:lmo_boundary_and_projection}.

        \item 
        This follows from \eqref{eq:lmo_step_has_length_t_combined}, part~\ref{it:bdry-sphere} of \Cref{thm:lmo_boundary_and_projection}.
        Further, by part \ref{it:bdry-proj} of the same theorem, $x_{\rm lmo} = \Proj_{\cX\cap\cH_{\rm lmo}}(x)$, so $\dist\rbr{x,\cX\cap\cH_{\rm lmo}} = \norm{x-x_{\rm lmo}} = t$.

        \item 
        \Cref{thm:projection_fejer_halfspace} applied with $\alpha = \alpha_{\rm lmo}(x,t) \ge \alpha_\star(x)$, which holds by part~\ref{it:lmoproj-localizes}, combined with the step length in part~\ref{it:lmoproj-stepsize}.

        Since the minimizer is identified with a projection, it is unique: 
        by part~\ref{it:lmoproj-stepsize}, every element of $\argmin_{z\in\cX \cap\cB(x,t)}\inner{g}{z}$ lies on the sphere $\norm{z-x} = t$, and that argmin is convex, so it is a singleton.
    \end{enumerate}
\end{proof}

The following one-step consequences are used repeatedly in \Cref{app:band,app:localrate}.

\begin{lemma}
\label{lem:descent} 
    Let $\{x_k\}$ be generated by \algname{Local LMO} with admissible radii $0 < t_k \le t(x_k)$ from $x_0 \in \cX$.
    Then, for every $k \ge 0$ and every $x_\star \in \cX_\star$,
    \begin{enumerate}
        \item \label{it:descent-1}
        $t_k \le \norm{x_k-x_\star}$,

        \item \label{it:descent-2}
        $\norm{x_{k+1}-x_k} = t_k$,

        \item \label{it:descent-3}
        $\norm{x_{k+1}-x_\star}^2 \le \norm{x_k-x_\star}^2-t_k^2$.
    \end{enumerate}
\end{lemma}

\begin{proof}
    Let $g_k = \nabla f(x_k)$.
    Convexity and Cauchy--Schwarz give $f(x_k) - f_\star \le \inner{g_k}{x_k-x_\star} \le \norm{g_k}\norm{x_k-x_\star}$, so $t_k \le t(x_k) = (f(x_k) - f_\star) / \norm{g_k} \le \norm{x_k-x_\star}$, which is part~\ref{it:descent-1}.
    Parts \ref{it:descent-2} and \ref{it:descent-3} correspond to parts~\ref{it:lmoproj-stepsize} and~\ref{it:lmoproj-distance} of \Cref{thm:lmo_is_projection}, respectively.
\end{proof}

\section{Proof of \Cref{thm:universal_band}: a universal rate on the deep-cut band}
\label{app:band}

\paragraph{Proof outline.}
\Cref{thm:universal_band} states that every projection method whose cut lies in the deep-cut band attains the rate $K^{-(1+\vartheta)/2}$, without knowledge of $\vartheta$ or $L_\vartheta$.
Our argument is based on two observations.
\begin{enumerate}
    \item
    \Cref{lem:polyak_cut} shows that a point $x_+$ in the Polyak half-space $\cH(\nabla f(x),\alpha_{\rm cvx}(x))$ satisfies $f(x_+)-f_\star \le D_f(x_+,x)$.
    Cutting at or below the Polyak level converts the Bregman divergence into a function gap.

    \item
    Fejér monotonicity (\Cref{thm:projection_fejer_halfspace}) makes the squared step lengths summable, so the shortest of $K$ steps is at most $\dist(x_0,\cX_\star)/\sqrt{K}$.
\end{enumerate}

Hölder continuity turns the first into $f(x_{k+1})-f_\star \le \frac{L_\vartheta}{1+\vartheta}\norm{x_{k+1}-x_k}^{1+\vartheta}$, and the second then gives the rate.
We stress that neither step uses $\vartheta$ or $L_\vartheta$, which is why the bound is universal.

\Cref{thm:alpha_star_exact,thm:projection_fejer_halfspace} of \Cref{app:identification} are re-used below.

\subsection{The Polyak cut turns a Bregman divergence into a function gap}
\label{sub:polyak_cut}

\begin{lemma}
\label{lem:polyak_cut}
    Fix $x \in \cX$ and let $x_+ \in \cH(\nabla f(x), \alpha_{\rm cvx}(x))$.
    Then
    \begin{equation}
        \label{eq:special_Polyak_ineq}
        f(x_+) - f_\star \le D_f(x_+,x) .
    \end{equation}
\end{lemma}

\begin{proof}
    By the definition \eqref{eq:alpha_cvx_def} of $\alpha_{\rm cvx}$ and the assumption $x_+ \in \cH(\nabla f(x), \alpha_{\rm cvx}(x))$,
    \begin{equation*}
        \inner{\nabla f(x)}{x_+} \le \alpha_{\rm cvx}(x) = \inner{\nabla f(x)}{x} - f(x) + f_\star .
    \end{equation*}
    Hence, the Bregman divergence $D_f(x_+,x)$ can be bounded below as
    \begin{align*}
        D_f(x_+,x)
        & = f(x_+) - f(x) - \inner{\nabla f(x)}{x_+-x}\\
        & = f(x_+) - f_\star + \rbr{f_\star - f(x) + \inner{\nabla f(x)}{x} - \inner{\nabla f(x)}{x_+}} \\
        & \ge f(x_+) - f_\star .
    \end{align*}
\end{proof}

\subsection{A universal rate for any method that cuts at the Polyak level}
\label{sub:universal_engine}

\begin{lemma}
\label{lem:universal_engine}
    Suppose $\nabla f$ is Hölder continuous with exponent $\vartheta \in [0,1]$ and constant $L_\vartheta$, in the sense that
    \begin{equation}
        \label{eq:band_holder}
        D_f(u,v) \le \frac{L_\vartheta}{1+\vartheta}\norm{u-v}^{1+\vartheta}, \qquad \forall u,v \in \R^d .
    \end{equation}
    Fix $x_0 \in \cX$ and let $\{x_k\}_{k \ge 0} \subset \cX$ be generated by any method with the following two properties:
    \begin{enumerate}
        \item
        \label{item:band_cut}
        $x_{k+1} \in \cH(\nabla f(x_k), \alpha_{\rm cvx}(x_k))$ for all $k \ge 0$;

        \item
        \label{item:band_fejer}
        $\norm{x_{k+1} - x_\star}^2 \le \norm{x_k - x_\star}^2 - \norm{x_{k+1} - x_k}^2$ for all $k \ge 0$ and all $x_\star \in \cX_\star$.
    \end{enumerate}
    Then, for every $K \ge 1$,
    \begin{equation}
        \label{eq:uy08fydo87fdo8f}
        \min_{0 \le k \le K-1} \rbr{f(x_{k+1})-f_\star} \le \frac{L_\vartheta}{1+\vartheta} \rbr{\frac{\dist^2(x_0,\cX_\star)}{K}}^{\frac{1+\vartheta}{2}} .
    \end{equation}
\end{lemma}

\begin{proof}
    By property \ref{item:band_cut}, \Cref{lem:polyak_cut} applies at $x = x_k$ with $x_+ = x_{k+1}$.
    Combining \eqref{eq:special_Polyak_ineq} with the Hölder bound \eqref{eq:band_holder}, used with $u = x_{k+1}$ and $v = x_k$, gives
    \begin{equation}
        \label{eq:0980d9gfg_89gyfg}
        f(x_{k+1}) - f_\star \le D_f(x_{k+1},x_k) \le \frac{L_\vartheta}{1+\vartheta} \norm{x_{k+1}-x_k}^{1+\vartheta} .
    \end{equation}
    Property \ref{item:band_fejer} telescopes as in the proof of \Cref{thm:projection_fejer_halfspace} and yields
    \begin{equation}
        \label{eq:1/K_ineq}
        \min_{0 \le k \le K-1} \norm{x_{k} - x_{k+1}}^2 \le \frac{\dist^2(x_0,\cX_\star)}{K}, \qquad \forall K \ge 1 .
    \end{equation}
    Let $k^\star \in \argmin_{0 \le k \le K-1}\norm{x_{k+1}-x_k}$. 
    Applying \eqref{eq:0980d9gfg_89gyfg} at $k = k^\star$ and then \eqref{eq:1/K_ineq},
    \begin{equation*}
        \begin{aligned}
            \min_{0 \le k \le K-1}\rbr{f(x_{k+1})-f_\star} \ &\le \ f(x_{k^\star+1})-f_\star \ \le \ \frac{L_\vartheta}{1+\vartheta}\norm{x_{k^\star+1}-x_{k^\star}}^{1+\vartheta}\\
            &\le \ \frac{L_\vartheta}{1+\vartheta}\rbr{\frac{\dist^2(x_0,\cX_\star)}{K}}^{\frac{1+\vartheta}{2}} ,
        \end{aligned}
    \end{equation*}
    which is \eqref{eq:uy08fydo87fdo8f}.
\end{proof}

\subsection{Proof of \Cref{thm:universal_band}}
\label{sub:proof_band}

\universalbandthm*

\begin{proof}
    Fix $K \ge 1$. 
    If $x_k \in \cX_\star$ for some $0 \le k \le K$, then $f(x_k) = f_\star$ and \eqref{eq:band_rate} holds trivially, so we may assume $x_k \in \cX\setminus\cX_\star$ for every $0 \le k \le K$. 

    We write $\alpha_k \eqdef \inner{g_k}{x_k}-s_k\norm{g_k}$ for the level of the $k$-th cut, so that $\cH_{x_k}(g_k,s_k) = \cH(g_k,\alpha_k)$ by \eqref{eq:dictionary} and the hypothesis $t(x_k) \le s_k \le s_\star(x_k)$ reads $\alpha_\star(x_k) \le \alpha_k \le \alpha_{\rm cvx}(x_k)$ by \eqref{eq:level_depth_table}.
    We verify the two hypotheses of \Cref{lem:universal_engine}.
    The lower bound $\alpha_k \ge \alpha_\star(x_k)$ gives $\cX_\star \subseteq \cH(g_k,\alpha_k)$ by \Cref{thm:alpha_star_exact}, so \Cref{thm:projection_fejer_halfspace} applies and the iterates satisfy property \ref{item:band_fejer}.
    The upper bound $\alpha_k \le \alpha_{\rm cvx}(x_k)$ gives $\cH(g_k,\alpha_k) \subseteq \cH(g_k,\alpha_{\rm cvx}(x_k))$, and since $x_{k+1} \in \cX \cap \cH(g_k,\alpha_k)$ by construction, property \ref{item:band_cut} holds as well.
    \Cref{lem:universal_engine} then gives \eqref{eq:uy08fydo87fdo8f}.
    Its minimum is taken over the $K$ iterates $x_1,\dots,x_K$ produced by the method, so it is sharper than \eqref{eq:band_rate}, whose minimum ranges over $x_0,\dots,x_K$ as well; 
    we state the weaker form in the main text for consistency among \eqref{eq:band_rate}, \eqref{eq:localrate} and \eqref{eq:rate_zero}.
\end{proof}

\section{Proof of \Cref{thm:localrate}: a universal rate for \algname{Local LMO}}
\label{app:localrate}

\paragraph{Proof outline.}
\algname{Local LMO} cuts above the Polyak level $\alpha_{\rm cvx}(x)$, so the analysis of \Cref{app:band} does not apply.
For convenience, we define the suboptimality gap of iterate $x_k$ as
\begin{equation}
    a_k \eqdef f(x_k) - f_\star .
\end{equation}

We further define the projection of the initial iterate onto the set of minimizers
\begin{equation}
    \label{eq:xstar0}
    x_\star^0 \eqdef \Proj_{\cX_\star}(x_0) ,
\end{equation}
and write $r_k \eqdef \norm{x_k-x_\star^0}$, so that $r_0 = R \eqdef \dist(x_0,\cX_\star)$.
We also write $G_\star \eqdef \norm{\nabla f(x_\star)}$, which by \Cref{lem:grad_const_on_Xstar} does not depend on the choice of $x_\star \in \cX_\star$.

Our proof proceeds in four steps.
\begin{enumerate}
    \item
    \label{item:ingr_holder}
    In \Cref{sub:holder}, we establish the conjugate Hölder inequality to give the gradient bound $\norm{g_k-\nabla f(x_\star)} \le \kappa a_k^{\vartheta/(1+\vartheta)}$ of \Cref{lem:grad_bound}---whence, by the triangle inequality, the two-sided bound $\bigl|\norm{g_k}-G_\star\bigr| \le \kappa a_k^{\vartheta/(1+\vartheta)}$ used in \eqref{eq:twosided}---and, with Fejér monotonicity, a priori bounds that require no compactness.

    \item
    \Cref{sub:gap} tackles the local Frank--Wolfe gap and the overshoot.
    With $\ell_k \eqdef \inner{g_k}{x_k-x_{k+1}}$ and $\rho_k \eqdef \ell_k/a_k \in (0,1]$, the descent inequality reads $a_{k+1} \le a_k(1-\rho_k+\delta_k)$ with $\delta_k = c\,t_k^{1+\vartheta}/a_k$, and the overshoot $\Gamma_k = (1-\rho_k) a_k$ provides a floor on the next function value.

    \item
    In \Cref{sub:budget}, we use the projection identity of \Cref{app:identification} to tighten $r_{k+1}^2 \le r_k^2-t_k^2$ to $r_{k+1}^2 \le r_k^2-t_k^2(2-\rho_k)/\rho_k$, whence $\sum_{k<K}t_k^2/\rho_k \le R^2$.
    Here, $\rho_k \in (0,1]$ is the fraction of the available first-order decrease captured by the linear minimization oracle; 
    at the Polyak radius $t_k = t(x_k)$ it is the cosine of the angle between the gradient $\nabla f(x_k)$ and the step $x_k - x_{k+1}$.
    The function value can fall at most a $\rho_k$ fraction (\Cref{lem:overshoot_lower_bound}), while the squared distance falls by $t_k^2 / \rho_k$.

    \item
    \Cref{sub:proof6} assembles the proof of \Cref{thm:localrate}.
    Let $A = \min_{k\le K} a_k$.
    If $A$ is above the level $\bar a$ at which the Hölder part of the gradient dominates $G_\star$, the Fejér budget alone gives the exponent $\tfrac{1+\vartheta}2$.
    Below $\bar a$ the two-sided bound of step \ref{item:ingr_holder} pins $\norm{g_k}$ to within a constant factor of $G_\star$, and the $K$ steps are counted three ways \citep{cartis2011adaptive,cartis2012evaluation}: 
    those above $\bar a$ (few, by the budget); 
    those that stall, $\rho_k < 2\delta_k$ (few, because each spends $a_kt_k^{1-\vartheta}/2c$ of the budget); 
    and the productive rest, counted by Cauchy--Schwarz against the potential $\Phi_k = a_k^{-2}$, using the identity \eqref{eq:gaincost}.
\end{enumerate}

\subsection{Two consequences of Hölder continuity}
\label{sub:holder}

Throughout this section $\nabla f$ is Hölder continuous with exponent $\vartheta \in [0,1]$ and constant $L_\vartheta$ as in \eqref{eq:holder}, and we write
\begin{equation}
    \label{eq:kappa_h}
    c \eqdef \frac{L_\vartheta}{1+\vartheta},\quad \kappa_\vartheta \eqdef \rbr{\frac{1+\vartheta}{\vartheta}}^{\frac{\vartheta}{1+\vartheta}},\quad \kappa \eqdef \kappa_\vartheta L_\vartheta^{\frac1{1+\vartheta}},\quad h(a) \eqdef \frac{a}{G_\star+\kappa\,a^{\vartheta/(1+\vartheta)}}\quad(a \ge 0),
\end{equation}
with $h(0) \eqdef 0$ when $G_\star = 0$.
At the endpoint $\vartheta = 0$ we read $\kappa_\vartheta$ as its limit, $\kappa_0 \eqdef \lim_{\vartheta\downarrow0}\kappa_\vartheta = 1$.
With this convention $\kappa = L_0$ and $h(a) = a/(G_\star+L_0)$ at $\vartheta = 0$, and all constants below are finite at both endpoints.
\Cref{cor:apriori} shows that $h$ is increasing and that $t(x_k) \ge h(a_k)$ along any run.

The first lemma of this subsection is the Hölder analogue of the standard smooth inequality $D_f(u,v) \ge \frac1{2L}\norm{\nabla f(u)-\nabla f(v)}^2$. 

\begin{lemma}
\label{lem:conj_holder}
    Let $f:\R^d \to \R$ be convex and differentiable with $D_f(u,v) \le \frac{L_\vartheta}{1+\vartheta}\norm{u-v}^{1+\vartheta}$ for all $u,v$, where $\vartheta \in (0,1]$.
    Then
    \begin{equation}
        \label{eq:conj_holder}
        D_f(u,v)\  \ge \ \frac{\vartheta}{1+\vartheta}\,L_\vartheta^{-1/\vartheta}\, \norm{\nabla f(u)-\nabla f(v)}^{\frac{1+\vartheta}{\vartheta}}\qquad\forall u,v \in \R^d .
    \end{equation}
\end{lemma}

\begin{proof}
    Fix $v$ and set $\varphi(w) \eqdef D_f(w,v) = f(w)-f(v)-\inner{\nabla f(v)}{w-v}$.
    Then $\varphi$ is convex and differentiable with $\nabla\varphi(w) = \nabla f(w)-\nabla f(v)$ and $\varphi \ge 0$ by convexity of $f$.
    Since $\varphi = f-\ell$ with $\ell(w) \eqdef f(v) + \inner{\nabla f(v)}{w-v}$ affine, and $D_\ell \equiv 0$ for affine $\ell$ because $\nabla\ell$ is constant, we have $D_\varphi = D_f-D_\ell = D_f$.
    Therefore, $\varphi$ obeys the same Hölder bound, with the same $L_\vartheta$.
    
    Fix $u$ and put $\Delta \eqdef \norm{\nabla\varphi(u)}$.
    If $\Delta = 0$, there is nothing to prove.
    
    Otherwise, let, for $s > 0$, $w \eqdef u - s \nabla\varphi(u) / \Delta$.
    Then
    \begin{equation*}
        0 \le \varphi(w) = \varphi(u) + \inner{\nabla\varphi(u)}{w-u}+D_\varphi(w,u) \le \varphi(u) - s \Delta + \frac{L_\vartheta}{1+\vartheta} s^{1+\vartheta}.
    \end{equation*}
    Choosing $s = (\Delta/L_\vartheta)^{1/\vartheta}$ makes $\frac{L_\vartheta}{1+\vartheta}s^{1+\vartheta} = \frac{L_\vartheta s^\vartheta}{1+\vartheta}\,s = \frac{s\Delta}{1+\vartheta}$, so $0 \le \varphi(u)-\frac{\vartheta}{1+\vartheta}s\Delta$, which is \eqref{eq:conj_holder}.
\end{proof}

The conjugate Hölder inequality is vacuous at $\vartheta = 0$. 
The following lemma supplies a Hölder-$0$ bound on the Bregman divergence which forces the gradient map to have bounded range.

\begin{lemma}
\label{lem:grad_bound_zero}
    Let $f$ satisfy $D_f(u,v) \le L_0\norm{u-v}$ for all $u,v \in \R^d$.
    Then
    \begin{equation}
        \label{eq:grad_bound_zero}
        \norm{\nabla f(u)-\nabla f(v)}\  \le \ L_0, \qquad \forall u,v \in \R^d .
    \end{equation}
\end{lemma}

\begin{proof}
    Fix $v$ and set $\varphi(w) \eqdef D_f(w,v)$, so that $\varphi$ is convex and differentiable with $\nabla\varphi(w) = \nabla f(w)-\nabla f(v)$, $\varphi \ge 0$, and $D_\varphi = D_f$, exactly as in the proof of \Cref{lem:conj_holder}.
    Fix $u$ and put $\Delta \eqdef \norm{\nabla\varphi(u)}$; if $\Delta = 0$ there is nothing to prove.
    For $s > 0$ let $w_s \eqdef u-s\nabla\varphi(u)/\Delta$, so that $\norm{w_s-u} = s$ and
    \begin{equation*}
        0\  \le \ \varphi(w_s)\  = \ \varphi(u)+\inner{\nabla\varphi(u)}{w_s-u}+D_\varphi(w_s,u)\  \le \ \varphi(u)-s\Delta+L_0\,s .
    \end{equation*}
    Hence $s\rbr{\Delta-L_0} \le \varphi(u)$ for every $s > 0$.
    Were $\Delta > L_0$, the left-hand side would be unbounded above, a contradiction; so $\Delta \le L_0$.
\end{proof}

\begin{lemma}
\label{lem:grad_bound}
    Let $\nabla f$ be Hölder continuous with exponent $\vartheta \in [0,1]$ and constant $L_\vartheta$.
    Then, for every $x \in \cX$ and every $x_\star \in \cX_\star$,
    \begin{equation}
        \label{eq:grad_bound}
        \norm{\nabla f(x)-\nabla f(x_\star)}\  \le \ \kappa_\vartheta\,L_\vartheta^{\frac{1}{1+\vartheta}}\, \rbr{f(x)-f_\star}^{\frac{\vartheta}{1+\vartheta}} .
    \end{equation}
\end{lemma}

\begin{proof}
    First-order optimality of $x_\star$ over $\cX$ gives $\inner{\nabla f(x_\star)}{x-x_\star} \ge 0$ for all $x \in \cX$, whence $D_f(x,x_\star) = f(x)-f_\star-\inner{\nabla f(x_\star)}{x-x_\star} \le f(x)-f_\star$.
    Combining with \Cref{lem:conj_holder} applied at $(u,v) = (x,x_\star)$, we find
    \begin{equation*}
        \frac{\vartheta}{1+\vartheta} L_\vartheta^{-1/\vartheta} \norm{\nabla f(x)-\nabla f(x_\star)}^{(1+\vartheta)/\vartheta} \le f(x) - f_\star .
    \end{equation*}
    Raising to the power $\vartheta/(1+\vartheta)$ and rearranging gives \eqref{eq:grad_bound}.
    For $\vartheta = 0$, \eqref{eq:grad_bound} reads $\norm{\nabla f(x)-\nabla f(x_\star)} \le \kappa_0L_0 = L_0$, which is \Cref{lem:grad_bound_zero}.
\end{proof}

\begin{corollary}
\label{cor:apriori}
    Along any \algname{Local LMO} run with $0 < t_k \le t(x_k)$ starting from $x_0 \in \cX$, we have
    \begin{align}
    \label{eq:apriori}
        a_k & \le a_{\max} \eqdef \max\curlybr{2G_\star R,\ (2\kappa_\vartheta)^{1+\vartheta}L_\vartheta R^{1+\vartheta}} \\
        \norm{g_k} & \le G \eqdef G_\star + \kappa_\vartheta L_\vartheta^{\frac1{1+\vartheta}} a_{\max}^{\frac{\vartheta}{1+\vartheta}} < +\infty ,
    \end{align}
    and consequently
    \begin{equation}
        \label{eq:polyak_radius_lb}
        t(x_k) = \frac{a_k}{\norm{g_k}} \ge \frac{a_k}{G_\star+\kappa\,a_k^{\vartheta/(1+\vartheta)}} = h(a_k),
    \end{equation}
    with $h$ as in \eqref{eq:kappa_h}, and $h$ is strictly increasing on $[0,+\infty)$.
\end{corollary}

\begin{proof}
    Convexity and the Cauchy--Schwarz inequality give
    \begin{equation*}
        a_k \le \inner{g_k}{x_k-x_\star^0} \le \norm{g_k}r_k \le \norm{g_k} R ,
    \end{equation*}
    where the last step follows from Fejér monotonicity (part~\ref{it:descent-3} of \Cref{lem:descent}).
    With \eqref{eq:grad_bound}, 
    \begin{equation}
        a_k \le G_\star R + \kappa_\vartheta L_\vartheta^{1/(1+\vartheta)}R \, a_k^{\vartheta/(1+\vartheta)} .
    \end{equation}
    
    Now, if $a \le u+va^{p}$ with $p = \frac{\vartheta}{1+\vartheta} \in [0,\tfrac{1}{2}]$ and $u,v \ge 0$, then either $a \le 2u$ or $a \le 2va^{p}$, i.e. $a \le (2v)^{1/(1-p)} = (2v)^{1+\vartheta}$.
    With $u = G_\star R$ and $v = \kappa R$ this is the first bound in \eqref{eq:apriori}, and the second follows from \eqref{eq:grad_bound} again.

    Equation \eqref{eq:polyak_radius_lb} is just \eqref{eq:grad_bound} together with $\norm{g_k} \le G_\star+\norm{g_k-\nabla f(x_\star)}$ for any $x_\star \in \cX_\star$.

    Finally $h(a) = a/(G_\star+\kappa a^{p})$ with $p = \tfrac{\vartheta}{1+\vartheta} \in [0,\tfrac12]$ has $h'(a) = \rbr{G_\star+(1-p)\kappa a^{p}}/(G_\star+\kappa a^{p})^2 > 0$ for $a > 0$;
    if $G_\star = 0$ this reads $h(a) = a^{1/(1+\vartheta)}/\kappa$, which is continuous at $0$ with our convention that $h(0) = 0$, so $h$ is increasing on all of $[0,\infty)$ in either case.
\end{proof}

\subsection{The local Frank--Wolfe gap and the overshoot}
\label{sub:gap}

\begin{definition}
\label{def:local_fw_gap}
    For $x \in \cX\setminus\cX_\star$ and $t > 0$, define
    \begin{equation*}
        \ell(x,t) \eqdef \inner{\nabla f(x)}{x}-\min_{z\in\cX\cap\cB(x,t)}\inner{\nabla f(x)}{z} = \inner{\nabla f(x)}{x-x_{\rm lmo}}\  \ge 0 .
    \end{equation*}
\end{definition}
\begin{lemma}
\label{lem:local_fw_gap}
    Fix $x \in \cX\setminus\cX_\star$ and $0 < t \le t(x)$.
    Then for every $x_\star \in \cX_\star$,
    \begin{equation}
        \label{eq:local_fw_gap}
        \ell(x,t)\  \ge \ \frac{t}{\norm{x-x_\star}}\,\inner{\nabla f(x)}{x-x_\star}\  \ge \ \frac{t}{\norm{x-x_\star}}\,\rbr{f(x)-f_\star}.
    \end{equation}
\end{lemma}

\begin{proof}
    Let $g = \nabla f(x)$.
    As in steps~2 and~3 of the proof of \Cref{thm:lmo_boundary_and_projection}, $0 < t \le t(x) \le \inner{g}{x-x_\star}/\norm{g} \le \norm{x-x_\star}$, so $\lambda \eqdef t/\norm{x-x_\star} \in (0,1]$ and $z_t \eqdef x-\lambda(x-x_\star) = (1-\lambda)x+\lambda x_\star \in \cX$ with $\norm{z_t-x} = t$; hence $z_t \in \cX\cap\cB(x,t)$.
    Therefore
    \begin{equation*}
        \min_{z\in\cX\cap\cB(x,t)}\inner{g}{z} \le \inner{g}{z_t} = \inner{g}{x}-\lambda\inner{g}{x-x_\star},
    \end{equation*}
    so $\ell(x,t) \ge \lambda\inner{g}{x-x_\star}$.
    The last inequality in \eqref{eq:local_fw_gap} is first-order convexity, $\inner{g}{x-x_\star} \ge f(x)-f_\star$.
\end{proof}
\begin{lemma}
\label{lem:overshoot}
    Fix $x_k \in \cX\setminus\cX_\star$ and $0 < t_k \le t(x_k)$, and let $x_{k+1}$ be a \algname{Local LMO} step.
    With $\ell_k \eqdef \ell(x_k,t_k)$,
    \begin{equation}
        \label{eq:overshoot}
        \inner{g_k}{x_{k+1}} = \alpha_{\rm lmo}(x_k,t_k) = \alpha_{\rm cvx}(x_k)+\Gamma_k, \qquad \Gamma_k \eqdef a_k-\ell_k = \alpha_{\rm lmo}(x_k,t_k)-\alpha_{\rm cvx}(x_k),
    \end{equation}
    and $\Gamma_k$ satisfies
    \begin{equation}
        \label{eq:overshoot_bounds}
        0\  \le \ \alpha_{\rm proj}(x_k,t_k)-\alpha_{\rm cvx}(x_k)\  \le \ \Gamma_k\  \le \ a_k\rbr{1-\tfrac{t_k}{\norm{x_k-x_\star}}} , \qquad \forall x_\star \in \cX_\star.
    \end{equation}
    Consequently $x_{k+1} \in \cH(g_k,\alpha_{\rm proj}(x_k,t_k))$ if and only if $\alpha_{\rm lmo}(x_k,t_k) = \alpha_{\rm proj}(x_k,t_k)$, i.e. iff the unconstrained ball minimizer $x_k-t_k\,g_k/\norm{g_k}$ is feasible in $\cX$.
    Whenever the constraint $\cX$ is active at the step, $x_{k+1} \notin \cH(g_k,\alpha_{\rm proj}(x_k,t_k))$.
\end{lemma}

\begin{proof}
    By \Cref{thm:lmo_boundary_and_projection}, $\inner{g_k}{x_{k+1}} = \alpha_{\rm lmo}(x_k,t_k)$.
    By definition $\alpha_{\rm cvx}(x_k) = \inner{g_k}{x_k}-a_k$, so $\alpha_{\rm lmo}(x_k,t_k)-\alpha_{\rm cvx}(x_k) = \inner{g_k}{x_{k+1}}-\inner{g_k}{x_k}+a_k = a_k-\ell_k = \Gamma_k$, which is \eqref{eq:overshoot}.
    For \eqref{eq:overshoot_bounds}: the middle inequality $\Gamma_k \ge \alpha_{\rm proj}-\alpha_{\rm cvx} \ge 0$ is \Cref{lem:alpha_t_le_alpha_lmo} together with \Cref{lem:alpha_t_vs_alpha_cvx}; the Cauchy--Schwarz bound $\ell_k = \inner{g_k}{x_k-x_{k+1}} \le \norm{g_k}\,\norm{x_k-x_{k+1}} = t_k\norm{g_k} \le t(x_k)\norm{g_k} = a_k$ gives $\Gamma_k \ge 0$; and the upper bound is \Cref{lem:local_fw_gap}.
    The final claim is \eqref{eq:overshoot}: $x_{k+1} \in \cH(g_k,\alpha_{\rm proj}) \Leftrightarrow \alpha_{\rm lmo} \le \alpha_{\rm proj} \Leftrightarrow \alpha_{\rm lmo} = \alpha_{\rm proj}$, using $\alpha_{\rm lmo} \ge \alpha_{\rm proj}$ (\Cref{lem:alpha_t_le_alpha_lmo}).
    Equality holds iff the minimum of $\inner{g_k}{\cdot}$ over $\cB(x_k,t_k)$, attained at $x_k-t_k g_k/\norm{g_k}$, is attained inside $\cX$.
\end{proof}
\begin{lemma}
\label{lem:overshoot_lower_bound}
    Under the assumptions of \Cref{lem:overshoot}, every \algname{Local LMO} step satisfies the exact identity $f(x_{k+1})-f_\star = D_f(x_{k+1},x_k)+\Gamma_k$, and hence, since $D_f \ge 0$,
    \begin{equation}
        \label{eq:overshoot_lower_bound}
        f(x_{k+1})-f_\star\  \ge \ \Gamma_k\  = \ a_k-\ell_k\  \ge \ 0.
    \end{equation}
    No \algname{Local LMO} step can push the suboptimality below the overshoot $\Gamma_k$, however smooth $f$ is.
    By \eqref{eq:overshoot_bounds}, $\Gamma_k \ge \alpha_{\rm proj}(x_k,t_k)-\alpha_{\rm cvx}(x_k) = \norm{g_k}\rbr{t(x_k)-t_k}$, so the floor vanishes only if $t_k = t(x_k)$ \emph{and} $\alpha_{\rm lmo} = \alpha_{\rm proj}$, i.e. only if the radius is the full Polyak radius and the unconstrained ball minimizer is feasible --- in which case the step coincides with the Polyak projection step.
    A shorter radius leaves a positive floor even when $\cX$ is inactive.
\end{lemma}

\begin{proof}
    By the definition of the Bregman divergence and of $\ell_k$,
    \begin{equation*}
        D_f(x_{k+1},x_k) = f(x_{k+1})-f(x_k)-\inner{g_k}{x_{k+1}-x_k} = f(x_{k+1})-f(x_k)+\ell_k .
    \end{equation*}
    Adding and subtracting $f_\star$ and using $\Gamma_k = a_k-\ell_k$ \eqref{eq:overshoot} gives $f(x_{k+1})-f_\star = D_f(x_{k+1},x_k)+a_k-\ell_k = D_f(x_{k+1},x_k)+\Gamma_k$.
    Since $D_f(x_{k+1},x_k) \ge 0$ by convexity and $\Gamma_k \ge 0$ by \eqref{eq:overshoot_bounds}, we obtain \eqref{eq:overshoot_lower_bound}.
\end{proof}
\begin{lemma}
\label{lem:master_recursion}
    Suppose $\nabla f$ is Hölder continuous with exponent $\vartheta \in [0,1]$ (that is, $D_f(u,v) \le \frac{L_\vartheta}{1+\vartheta}\norm{u-v}^{1+\vartheta}$ for all $u,v$), and let $\{x_k\}$ be generated by \algname{Local LMO} with $0 < t_k \le t(x_k)$.
    With $R = \dist(x_0,\cX_\star)$, for every $k \ge 0$,
    \begin{equation}
        \label{eq:master_recursion}
        a_{k+1}\  \le \ a_k-\ell_k+\frac{L_\vartheta}{1+\vartheta}\,t_k^{\,1+\vartheta} \  \le \ \rbr{1-\tfrac{t_k}{R}}a_k+\frac{L_\vartheta}{1+\vartheta}\,t_k^{\,1+\vartheta}.
    \end{equation}
\end{lemma}

\begin{proof}
    The Hölder descent inequality with $u = x_{k+1}$, $v = x_k$ gives $a_{k+1} = a_k-\ell_k+D_f(x_{k+1},x_k) \le a_k-\ell_k+\frac{L_\vartheta}{1+\vartheta}\norm{x_{k+1}-x_k}^{1+\vartheta}$, and $\norm{x_{k+1}-x_k} = t_k$ by \Cref{thm:lmo_boundary_and_projection}; this is the first inequality.
    For the second, take $x_\star = \Proj_{\cX_\star}(x_0)$, so that part~\ref{it:descent-3} of \Cref{lem:descent} gives $\norm{x_k-x_\star} \le \norm{x_0-x_\star} = R$, so \Cref{lem:local_fw_gap} yields $\ell_k \ge (t_k/\norm{x_k-x_\star})a_k \ge (t_k/R)a_k$.
\end{proof}

\subsection{The sharpened Fejér budget}
\label{sub:budget}

Part~\ref{it:descent-3} of \Cref{lem:descent} established the inequality $r_{k+1}^2 \le r_k^2-t_k^2$.
Our next theorem shows that this is lossy when the overshoot $\Gamma_k$ takes up a large share of the gap $a_k$, and quantifies the gain using the projection identity of \Cref{thm:lmo_boundary_and_projection}.

\begin{theorem}
\label{thm:sharp_fejer}
    Let $x_k \in \cX \setminus \cX_\star$, let $0 < t_k \le t(x_k)$, and let $x_{k+1}$ be a \algname{Local LMO} step.
    Put $\ell_k = \inner{g_k}{x_k-x_{k+1}}\ ( > 0)$ and, for each $x_\star \in \cX_\star$,
    \begin{equation*}
        \varepsilon_k(x_\star) \eqdef \inner{g_k}{x_{k+1}-x_\star} \ge \Gamma_k = a_k-\ell_k \ge 0 .
    \end{equation*}
    Then, for every $x_\star \in \cX_\star$,
    \begin{equation}
        \label{eq:sharp_fejer}
        \norm{x_{k+1}-x_\star}^2 \le \norm{x_k-x_\star}^2 - t_k^2\rbr{1+\frac{2\varepsilon_k(x_\star)}{\ell_k}} \le \norm{x_k-x_\star}^2 - t_k^2\rbr{1+\frac{2\Gamma_k}{\ell_k}}.
    \end{equation}
\end{theorem}

\begin{proof}
    \Cref{lem:local_fw_gap} shows that $\ell_k \ge (t_k/r_k)a_k > 0$, with $r_k = \norm{x_k-x_\star^0}$ as in \eqref{eq:xstar0}.
    
    To establish $\varepsilon_k(x_\star) \ge \Gamma_k \ge 0$, note that convexity gives $\inner{g_k}{x_k-x_\star} \ge a_k$, i.e. $\inner{g_k}{x_\star} \le \alpha_{\rm cvx}(x_k)$. 
    Hence, $\varepsilon_k(x_\star) = \alpha_{\rm lmo}(x_k,t_k)-\inner{g_k}{x_\star} \ge \alpha_{\rm lmo}(x_k,t_k)-\alpha_{\rm cvx}(x_k) = \Gamma_k$ by \Cref{thm:lmo_boundary_and_projection} and \eqref{eq:overshoot}, and $\Gamma_k \ge 0$ by \Cref{lem:overshoot}.

    Write $\alpha \eqdef \alpha_{\rm lmo}(x_k,t_k)$ and $\cH_k \eqdef \cH(g_k,\alpha)$, so that $\inner{g_k}{x_k} = \alpha+\ell_k$ and $\inner{g_k}{x_\star} = \alpha - \varepsilon_k(x_\star)$.
    Set
    \begin{equation*}
        s \eqdef \frac{\ell_k}{\ell_k+\varepsilon_k(x_\star)} \in (0,1] , \qquad z \eqdef (1-s) x_k + s x_\star .
    \end{equation*}
    Then $z \in \cX$ by convexity, and
    \begin{equation*}
        \inner{g_k}{z} = (1-s)(\alpha+\ell_k)+s(\alpha-\varepsilon_k(x_\star)) = \alpha+\ell_k-s(\ell_k+\varepsilon_k(x_\star)) = \alpha,
    \end{equation*}
    so $z \in \cX\cap\cH_k$.

    By \Cref{thm:lmo_boundary_and_projection}, $x_{k+1} = \Proj_{\cX\cap\cH_k}(x_k)$, so the projection inequality at the point $z \in \cX\cap\cH_k$ reads $\inner{x_k-x_{k+1}}{z-x_{k+1}} \le 0$.
    Since $z-x_{k+1} = (x_k-x_{k+1})-s(x_k-x_\star)$ and $\norm{x_k-x_{k+1}} = t_k$ (\Cref{thm:lmo_boundary_and_projection}), this is
    \begin{equation*}
        t_k^2 - s\,\inner{x_k-x_{k+1}}{x_k-x_\star} \le 0, \qquad\text{i.e.}\qquad \inner{x_k-x_{k+1}}{x_k-x_\star} \ge \frac{t_k^2}{s} = t_k^2\rbr{1+\frac{\varepsilon_k(x_\star)}{\ell_k}}.
    \end{equation*}
    Finally, $\norm{x_{k+1}-x_\star}^2 = \norm{x_k-x_\star}^2-2\inner{x_k-x_{k+1}}{x_k-x_\star}+t_k^2$, which gives \eqref{eq:sharp_fejer}. 
    The second inequality is $\varepsilon_k(x_\star) \ge \Gamma_k$.
\end{proof}

\begin{remark}
\label{rem:overshoot_pays}
    At the Polyak radius $t_k = t(x_k)$, with $x_k \in \cX\setminus\cX_\star$, one has $\ell_k \le \norm{g_k}t_k = a_k$, so, specializing \eqref{eq:sharp_fejer} to $x_\star = x_\star^0$ and writing $\rho_k \eqdef \ell_k/a_k \in (0,1]$, \eqref{eq:sharp_fejer} reads
    \begin{equation}
        \label{eq:sharp_fejer_polyak}
        r_{k+1}^2 \le r_k^2-t_k^2\,\frac{2-\rho_k}{\rho_k} \le r_k^2-\frac{t_k^2}{\rho_k}.
    \end{equation}
    \Cref{lem:overshoot_lower_bound} now says that the step cannot reduce the value below $\Gamma_k = (1-\rho_k) a_k$: the smaller $\rho_k$, the worse the value progress.
    On the other hand, \eqref{eq:sharp_fejer_polyak} says the smaller $\rho_k$, the better the distance progress, by precisely the reciprocal factor.
\end{remark}

\begin{corollary}
\label{cor:budget}
    For \algname{Local LMO} at the Polyak radius $t_k = t(x_k)$, with $x_k \in \cX\setminus\cX_\star$ for every $k < K$, for every $K \ge 1$,
    \begin{equation}
        \label{eq:budget}
        \sum_{k=0}^{K-1}\frac{t_k^2}{\rho_k} = \sum_{k=0}^{K-1}\frac{t_k^2\,a_k}{\ell_k} \le R^2 .
    \end{equation}
\end{corollary}

\begin{proof}
    This follows by telescoping \eqref{eq:sharp_fejer_polyak} and using $r_K^2 \ge 0$.
\end{proof}

\subsection{Proof of \Cref{thm:localrate}}
\label{sub:proof6}

The proof is a three-way accounting of the $K$ steps, run against two budgets: 
the sharpened Fejér budget $\sum_{k<K}t_k^2/\rho_k \le R^2$ (\Cref{cor:budget}) and the telescoping of the potential $\Phi_k \eqdef a_k^{-2}$.
What makes the two fit together is the following identity,
\begin{equation}
    \label{eq:gaincost}
    \underbrace{\frac{\rho_k}{a_k^2}}_{\text{potential gain}}\cdot\underbrace{\frac{t_k^2}{\rho_k}}_{\text{Fejér cost}} \  = \ \frac{t_k^2}{a_k^2}\  = \ \frac1{\norm{g_k}^2}.
\end{equation}
Gain times cost does not depend on $\rho_k$ at all.
A step that makes little progress ($\rho_k$ small) pays for it in the Fejér budget by the reciprocal factor, so no choice of $\rho_k$ is free, and Cauchy--Schwarz converts \eqref{eq:gaincost} into a bound on the number of steps.
The only obstruction is the Hölder error $c\,t_k^{1+\vartheta}$, which can undo the potential gain. 
It is controlled by the gradient bound of \Cref{lem:grad_bound}, which together with the triangle inequality pins $\norm{g_k}$ to within a factor of the constant $G_\star$ once $a_k$ is below an explicit level, and by the Fejér budget above that level.

Throughout this section, we set $\lambda \eqdef \tfrac{2}{3}$,
\begin{equation}
    \label{eq:abar}
    \bar a \eqdef \rbr{\frac{\lambda G_\star}{\kappa}}^{\frac{1+\vartheta}{\vartheta}} , \qquad\text{so that}\qquad \kappa\,\bar a^{\frac{\vartheta}{1+\vartheta}} = \lambda G_\star ,
\end{equation}
with the conventions $\bar a \eqdef 0$ if $G_\star = 0$ and $\bar a \eqdef +\infty$ if $L_\vartheta = 0 < G_\star$,
with $\kappa = \kappa_\vartheta L_\vartheta^{1/(1+\vartheta)}$ and $c = L_\vartheta/(1+\vartheta)$ as in \eqref{eq:kappa_h}, and
\begin{equation}
    \label{eq:delta_k}
    \delta_k \eqdef \frac{c\,t_k^{1+\vartheta}}{a_k} = \frac{c\,t_k^{\vartheta}}{\norm{g_k}}, \qquad\text{so that}\qquad a_{k+1} \le a_k\rbr{1-\rho_k+\delta_k}
\end{equation}
by \Cref{lem:master_recursion} and $a_k = t_k \norm{g_k}$.

Before we offer the proof of \Cref{thm:localrate}, we establish the bounds on the constants $D_\vartheta, E_\vartheta$ introduced next.

With $\Theta_\vartheta \eqdef 5^{(1+\vartheta)/2}$, define
\begin{equation}
    \label{eq:Dtheta_Etheta}
    \begin{aligned}
    D_\vartheta& \eqdef \rbr{\tfrac52\,\kappa_\vartheta}^{1+\vartheta},\qquad E_\vartheta \eqdef \max\curlybr{E^{(1)}_\vartheta,E^{(2)}_\vartheta,E^{(3)}_\vartheta},\\[2pt]
    E^{(1)}_\vartheta& \eqdef \sbr{\tfrac{25}{9}\,2^{\vartheta}\rbr{\tfrac32}^{1+\vartheta}\kappa_\vartheta^{1+\vartheta}}^{\frac1{2-\vartheta}}, \quad E^{(2)}_\vartheta \eqdef \sbr{\tfrac{8(1+\Theta_\vartheta)}{1+\vartheta}\rbr{\tfrac53}^{1-\vartheta}}^{\frac1{2-\vartheta}}, \quad E^{(3)}_\vartheta \eqdef \tfrac{20}3 .
    \end{aligned}
\end{equation}

\begin{lemma}
\label{lem:constants}
    For every $\vartheta \in [0,1]$, $D_\vartheta \le \tfrac{25}2$ and $E_\vartheta \le 25$, both attained at $\vartheta = 1$.
\end{lemma}

\begin{proof}
    Write $\psi(\vartheta) \eqdef \kappa_\vartheta^{1+\vartheta} = \rbr{\tfrac{1+\vartheta}{\vartheta}}^{\vartheta} = \exp\rbr{\vartheta\log(1+\tfrac1\vartheta)}$.
    Its exponent has derivative $\log(1+\tfrac1\vartheta)-\tfrac1{1+\vartheta} \ge 0$, because $\log(1+x) \ge \tfrac{x}{1+x}$ with $x = 1/\vartheta$ gives exactly $\log(1+\tfrac1\vartheta) \ge \tfrac{1/\vartheta}{1+1/\vartheta} = \tfrac1{1+\vartheta}$.
    So $\psi$ is nondecreasing with $\psi(\vartheta) \le \psi(1) = 2$.
    Hence $D_\vartheta = (\tfrac52)^{1+\vartheta}\psi(\vartheta) \le (\tfrac52)^2\cdot2 = \tfrac{25}2$, attained at $\vartheta = 1$.

    For $E_\vartheta$ note first that $\tfrac1{2-\vartheta} \le 1$, so a bracket that is $\ge 1$ is not increased by raising it to that power; this settles the first bracket, while the second needs the sharper bound below.
    The first bracket is $\tfrac{25}9\,2^{\vartheta}(\tfrac32)^{1+\vartheta}\psi(\vartheta) \le \tfrac{25}9\cdot2\cdot\tfrac94\cdot2 = 25$, with equality at $\vartheta = 1$, so $E^{(1)}_\vartheta \le 25$ and $E^{(1)}_1 = 25$.
    For the second, put $B(\vartheta) \eqdef 8\rbr{1+5^{(1+\vartheta)/2}}(\tfrac53)^{1-\vartheta}/(1+\vartheta)$, so $E^{(2)}_\vartheta = B(\vartheta)^{1/(2-\vartheta)}$, and it suffices to show $B(\vartheta) \le 25^{\,2-\vartheta}$.
    Since $5^{(1+\vartheta)/2} \le 5$ on $[0,1]$ we have $B(\vartheta) \le \widetilde B(\vartheta) \eqdef 48(\tfrac53)^{1-\vartheta}/(1+\vartheta)$, and
    \begin{equation*}
        \frac{d}{d\vartheta}\log\frac{25^{\,2-\vartheta}}{\widetilde B(\vartheta)} = \frac1{1+\vartheta}-\log25+\log\tfrac53\  \le \ 1-3.218+0.512\  < \ 0 ,
    \end{equation*}
    so $25^{2-\vartheta}/\widetilde B(\vartheta)$ is decreasing and attains its minimum on $[0,1]$ at $\vartheta = 1$, where it equals $25/48\cdot2 = \tfrac{50}{48} > 1$.
    Hence $E^{(2)}_\vartheta \le \widetilde B(\vartheta)^{1/(2-\vartheta)} \le 25$.
    Finally, $E^{(3)}_\vartheta = \tfrac{20}3 \le 25$.
\end{proof}

The endpoint $\vartheta = 0$ is settled separately, and more sharply, by the following elementary bound.

\begin{proposition}
\label{prop:rate_zero}
    Let $\vartheta = 0$, that is, $D_f(u,v) \le L_0\norm{u-v}$ for all $u,v \in \R^d$, and let $\{x_k\}$ be generated by \algname{Local LMO} at the Polyak radius $t_k = t(x_k)$ from $x_0 \in \cX$.
    Then, for every $K \ge 1$,
    \begin{equation}
        \label{eq:rate_zero}
        \min_{0\le k\le K} a_k\  \le \ \frac{\rbr{G_\star+L_0}R}{\sqrt K} .
    \end{equation}
\end{proposition}

\begin{proof}
    By \Cref{lem:grad_bound_zero} applied at $(u,v) = (x_k,x_\star)$ for any $x_\star \in \cX_\star$, $\norm{g_k} \le \norm{\nabla f(x_\star)}+\norm{g_k-\nabla f(x_\star)} \le G_\star+L_0$.
    At the Polyak radius $a_k = t_k\norm{g_k}$, so $a_k \le \rbr{G_\star+L_0}t_k$ for every $k$.
    By part~\ref{it:descent-3} of \Cref{lem:descent} and $r_0 = R$, telescoping gives $\sum_{k<K}t_k^2 \le r_0^2-r_K^2 \le R^2$, hence $\min_{k<K}t_k^2 \le R^2/K$.
    Choosing $k^\star \in \argmin_{k<K}t_k$ (so that $k^\star \le K-1$, and the bound in fact holds for $\min_{0\le k\le K-1}a_k$),
    \begin{equation*}
        \min_{0\le k\le K} a_k\  \le \ a_{k^\star}\  \le \ \rbr{G_\star+L_0}t_{k^\star}\  \le \ \frac{\rbr{G_\star+L_0}R}{\sqrt K} . \qedhere
    \end{equation*}
\end{proof}

\localratethm*

\begin{proof}
    Suppose first that $\vartheta = 0$, so that both exponents in \eqref{eq:localrate} equal $\tfrac12$ and the claim reads $\min_{k\le K}a_k \le \max\curlybr{\Lambda_1,\Lambda_2}K^{-1/2}$ with $\Lambda_1 = \tfrac{25}2L_0R$ and $\Lambda_2 \ge 25\,G_\star R$, by \eqref{eq:lambdas}.
    Since $u+v \le \max\curlybr{\tfrac{25}2u,25v}$ for all $u,v \ge 0$, \Cref{prop:rate_zero} yields the claim.
    Assume from now on that $\vartheta \in (0,1]$.

    Set $A \eqdef \min_{0\le k\le K} a_k$.
    If $A = 0$ there is nothing to prove, so assume $A > 0$; then $a_k \ge A > 0$ for every $k \le K$, and every $t_k$, $\rho_k$, $\delta_k$ below is well defined and finite.
    We use throughout: $\rho_k \in (0,1]$ (\Cref{lem:local_fw_gap} and \Cref{lem:overshoot}); $t_k \ge h(a_k) \ge h(A)$ with $h$ increasing (\Cref{cor:apriori}); $\sum_{k<K}t_k^2 \le \sum_{k<K}t_k^2/\rho_k \le R^2$ (\Cref{cor:budget}); and the \emph{two-sided} gradient bound
    \begin{equation}
        \label{eq:twosided}
        G_\star-\kappa\,a_k^{\frac{\vartheta}{1+\vartheta}}\  \le \ \norm{g_k}\  \le \ G_\star+\kappa\,a_k^{\frac{\vartheta}{1+\vartheta}},
    \end{equation}
    which is \eqref{eq:grad_bound} together with the triangle inequality in both directions.

    \emph{Case 1: $A > \bar a$.} 
    We may assume $\kappa > 0$: if $L_\vartheta = 0$ then $\kappa = 0$, and either $G_\star > 0$, in which case $\bar a = +\infty$ by the convention following \eqref{eq:abar} and this case is empty, or $G_\star = 0$, in which case $D_f \equiv 0$ forces $f$ to be affine with $\nabla f \equiv \nabla f(x_\star) = 0$, so $a_0 = 0$ and $A = 0$, already excluded.

    By \eqref{eq:abar} this means $\kappa A^{\vartheta/(1+\vartheta)} > \lambda G_\star$, hence $G_\star+\kappa A^{\vartheta/(1+\vartheta)} < \tfrac{1+\lambda}{\lambda}\kappa A^{\vartheta/(1+\vartheta)}$ and
    \begin{equation*}
        h(A)\  = \ \frac{A}{G_\star+\kappa A^{\vartheta/(1+\vartheta)}}\  > \ \frac{\lambda}{(1+\lambda)\kappa}\,A^{\frac1{1+\vartheta}} .
    \end{equation*}
    Since $t_k \ge h(A)$ for all $k < K$, the Fejér budget gives $K\,h(A)^2 \le R^2$, i.e. $A^{2/(1+\vartheta)} \le \rbr{\tfrac{(1+\lambda)\kappa}{\lambda}}^2R^2/K$, i.e.
    \begin{equation*}
        A\  \le \ \rbr{\tfrac{1+\lambda}{\lambda}\,\kappa R}^{1+\vartheta}K^{-\frac{1+\vartheta}2} \  = \ D_\vartheta\,L_\vartheta R^{1+\vartheta}K^{-\frac{1+\vartheta}2} \  \le \ \Lambda_1K^{-\frac{1+\vartheta}2},
    \end{equation*}
    using $\kappa^{1+\vartheta} = \kappa_\vartheta^{1+\vartheta}L_\vartheta$ and $\tfrac{1+\lambda}{\lambda} = \tfrac52$, so that the middle term is $\rbr{\tfrac52\kappa_\vartheta}^{1+\vartheta}L_\vartheta R^{1+\vartheta}K^{-(1+\vartheta)/2} = D_\vartheta L_\vartheta R^{1+\vartheta}K^{-(1+\vartheta)/2}$ by \eqref{eq:Dtheta_Etheta}, together with $D_\vartheta \le \tfrac{25}2$ from \Cref{lem:constants}. 
    (If $G_\star = 0$ then $\bar a = 0$ and this is the only case, so \eqref{eq:localrate} holds with $\Lambda_2 = 0$.)

    \emph{Case 2: $A \le \bar a$} (so $G_\star > 0$).
    Split $\{0,\dots,K-1\} = \cI\cup\cJ$,
    \begin{equation*}
        \cI \eqdef \{k < K:\ a_k \le \bar a\},\qquad \cJ \eqdef \{k < K:\ a_k > \bar a\} .
    \end{equation*}
    For $k \in \cI$, \eqref{eq:twosided} and $\kappa a_k^{\vartheta/(1+\vartheta)} \le \kappa\bar a^{\vartheta/(1+\vartheta)} = \lambda G_\star$ give the two-sided window
    \begin{equation}
        \label{eq:window}
        (1-\lambda)G_\star\  \le \ \norm{g_k}\  \le \ (1+\lambda)G_\star, \qquad\text{hence}\qquad \frac{a_k}{(1+\lambda)G_\star}\  \le \ t_k\  \le \ \frac{a_k}{(1-\lambda)G_\star}.
    \end{equation}
    For $k \in \cJ$ we have $t_k \ge h(a_k) > h(\bar a) = \bar a/\rbr{(1+\lambda)G_\star}$, so $\sum_{k<K}t_k^2 \le R^2$ forces
    \begin{equation}
        \label{eq:Kone}
        \abs{\cJ}\  \le \ K_1\  \eqdef \ \frac{(1+\lambda)^2G_\star^2R^2}{\bar a^{\,2}} .
    \end{equation}

    \emph{Case 2a: $K \le 4K_1$.} Since $A \le \bar a$ we have, as in \eqref{eq:window}, $h(A) \ge A/((1+\lambda)G_\star)$, so $K\,h(A)^2 \le R^2$ gives $A \le (1+\lambda)G_\star R\,K^{-1/2}$.
    As $\tfrac1{2-\vartheta}-\tfrac12 = \tfrac{\vartheta}{2(2-\vartheta)} \ge 0$ and $1 \le K \le 4K_1$,
    \begin{equation*}
        A\  \le \ (1+\lambda)G_\star R\,K^{-\frac1{2-\vartheta}}K^{\frac{\vartheta}{2(2-\vartheta)}} \  \le \ (1+\lambda)G_\star R\,\rbr{\frac{2(1+\lambda)G_\star R}{\bar a}}^{\frac{\vartheta}{2-\vartheta}}K^{-\frac1{2-\vartheta}},
    \end{equation*}
    using $(4K_1)^{1/2} = 2(1+\lambda)G_\star R/\bar a$.
    By \eqref{eq:abar}, $\bar a^{-\vartheta/(2-\vartheta)} = (\kappa/(\lambda G_\star))^{(1+\vartheta)/(2-\vartheta)}$, and since $1+\tfrac{\vartheta}{2-\vartheta} = \tfrac2{2-\vartheta}$ the prefactor equals
    \begin{equation*}
        (1+\lambda)^{\frac2{2-\vartheta}}2^{\frac{\vartheta}{2-\vartheta}}\lambda^{-\frac{1+\vartheta}{2-\vartheta}} \kappa_\vartheta^{\frac{1+\vartheta}{2-\vartheta}}\, \rbr{G_\star^{1-\vartheta}R^2L_\vartheta}^{\frac1{2-\vartheta}} \  = \ E^{(1)}_\vartheta\rbr{L_\vartheta R^2G_\star^{1-\vartheta}}^{\frac1{2-\vartheta}}\  \le \ \Lambda_2 ,
    \end{equation*}
    because $(G_\star R)^{2/(2-\vartheta)}G_\star^{-(1+\vartheta)/(2-\vartheta)} = (G_\star^{1-\vartheta}R^2)^{1/(2-\vartheta)}$ and $\kappa^{(1+\vartheta)/(2-\vartheta)} = (\kappa_\vartheta^{1+\vartheta}L_\vartheta)^{1/(2-\vartheta)}$; with $\lambda = \tfrac23$ the bracket is $(\tfrac53)^2 2^\vartheta(\tfrac32)^{1+\vartheta}\kappa_\vartheta^{1+\vartheta}$.
    The final inequality holds because $E^{(1)}_\vartheta \le E_\vartheta \le 25$ by \Cref{lem:constants}, and $\Lambda_2 \ge 25\rbr{L_\vartheta R^2G_\star^{1-\vartheta}}^{1/(2-\vartheta)}$ by \eqref{eq:lambdas}.

    \emph{Case 2b: $K > 4K_1$.} Split $\cI = \cI_P\cup\cI_S$ with
    \begin{equation*}
        \cI_P \eqdef \{k \in \cI:\ \rho_k \ge 2\delta_k\}\quad(\text{\emph{productive}}),\qquad \cI_S \eqdef \{k \in \cI:\ \rho_k < 2\delta_k\}\quad(\text{\emph{stalling}}).
    \end{equation*}

    \emph{(i) Stalling steps are few.} For $k \in \cI_S$, $\rho_k < 2\delta_k$ and \eqref{eq:delta_k} give
    \begin{equation*}
        \frac{t_k^2}{\rho_k}\  > \ \frac{t_k^2}{2\delta_k}\  = \ \frac{a_k\,t_k^{1-\vartheta}}{2c} \  \ge \ \frac{a_k^{2-\vartheta}}{2c\rbr{(1+\lambda)G_\star}^{1-\vartheta}} \  \ge \ \frac{A^{2-\vartheta}}{2c\rbr{(1+\lambda)G_\star}^{1-\vartheta}},
    \end{equation*}
    by the left inequality of \eqref{eq:window} and $a_k \ge A$.
    Summing and using \Cref{cor:budget},
    \begin{equation}
        \label{eq:NS}
        \abs{\cI_S}\  \le \ N\  \eqdef \ 2c\rbr{(1+\lambda)G_\star}^{1-\vartheta}R^2A^{\vartheta-2}.
    \end{equation}

    \emph{(ii) The potential.} Put $\Phi_k \eqdef a_k^{-2}$, well defined since $a_k \ge A > 0$.
    If $k \in \cI_P$ then $\rho_k \ge 2\delta_k$ and \eqref{eq:delta_k} give $a_{k+1} \le a_k(1-\tfrac{\rho_k}2)$ with $1-\tfrac{\rho_k}2 \ge \tfrac12 > 0$, so, by $(1-s)^{-2} \ge 1+2s$ on $[0,1)$ applied at $s = \rho_k/2$,
    \begin{equation}
        \label{eq:gainP}
        \Phi_{k+1}-\Phi_k\  \ge \ \Phi_k\sbr{(1-\tfrac{\rho_k}2)^{-2}-1}\  \ge \ \frac{\rho_k}{a_k^2}.
    \end{equation}
    If $k \in \cI_S$ then $a_{k+1} \le a_k(1+\delta_k)$, so by $1-(1+s)^{-2} \le 2s$ on $[0,\infty)$ and, by the right inequality of \eqref{eq:window}, $\delta_k = c\,t_k^{\vartheta}/\norm{g_k} \le c\,a_k^{\vartheta}/\rbr{(1-\lambda)^{1+\vartheta}G_\star^{1+\vartheta}}$,
    \begin{equation}
        \label{eq:lossS}
        \Phi_k-\Phi_{k+1}\  \le \ \frac{2\delta_k}{a_k^2}\  \le \ \frac{2c\,a_k^{\vartheta-2}}{(1-\lambda)^{1+\vartheta}G_\star^{1+\vartheta}} \  \le \ \frac{2c\,A^{\vartheta-2}}{(1-\lambda)^{1+\vartheta}G_\star^{1+\vartheta}},
    \end{equation}
    the last step because $\vartheta-2 < 0$ and $a_k \ge A$.
    If $k \in \cJ$ then simply $\Phi_k-\Phi_{k+1} \le \Phi_k = a_k^{-2} < \bar a^{-2}$.
    Telescoping over all $k < K$ and discarding $-\Phi_0 \le 0$,
    \begin{equation*}
        \sum_{k\in\cI_P}\rbr{\Phi_{k+1}-\Phi_k} \  = \ \Phi_K-\Phi_0-\!\!\sum_{k\in\cI_S\cup\cJ}\!\!\rbr{\Phi_{k+1}-\Phi_k} \  \le \ \Phi_K+\!\!\sum_{k\in\cI_S\cup\cJ}\!\!\rbr{\Phi_k-\Phi_{k+1}},
    \end{equation*}
    so, with $\Phi_K \le A^{-2}$, \eqref{eq:gainP}, \eqref{eq:lossS} and \eqref{eq:Kone},
    \begin{equation}
        \label{eq:potsum}
        \sum_{k\in\cI_P}\frac{\rho_k}{a_k^2}\  \le \ \frac1{A^2} \ +\ \frac{2c\,\abs{\cI_S}A^{\vartheta-2}}{(1-\lambda)^{1+\vartheta}G_\star^{1+\vartheta}} \ +\ \frac{K_1}{\bar a^{\,2}} .
    \end{equation}

    \emph{(iii) Productive steps are few.} By Cauchy--Schwarz, \Cref{cor:budget} and $t_k/a_k = 1/\norm{g_k} \ge 1/\rbr{(1+\lambda)G_\star}$ on $\cI$,
    \begin{equation*}
        \rbr{\frac{\abs{\cI_P}}{(1+\lambda)G_\star}}^{2} \le \rbr{\sum_{k\in\cI_P}\frac{t_k}{a_k}}^{2} \le \rbr{\sum_{k\in\cI_P}\frac{t_k^2}{\rho_k}}\rbr{\sum_{k\in\cI_P}\frac{\rho_k}{a_k^2}} \le R^2\sum_{k\in\cI_P}\frac{\rho_k}{a_k^2}.
    \end{equation*}
    Insert \eqref{eq:potsum} and use $\sqrt{x+y+z} \le \sqrt x+\sqrt y+\sqrt z$.
    The first term contributes $(1+\lambda)G_\star R/A$; the third contributes $(1+\lambda)G_\star R\sqrt{K_1}/\bar a = \rbr{(1+\lambda)G_\star R/\bar a}^2 = K_1$; and the second, using $\abs{\cI_S} \le N$ from \eqref{eq:NS}, contributes
    \begin{equation*}
        (1+\lambda)G_\star R\cdot\frac{\sqrt{2cNA^{\vartheta-2}}}{(1-\lambda)^{\frac{1+\vartheta}2}G_\star^{\frac{1+\vartheta}2}} = 2cR^2G_\star^{1-\vartheta}A^{\vartheta-2}\,\frac{(1+\lambda)^{\frac{3-\vartheta}2}}{(1-\lambda)^{\frac{1+\vartheta}2}} = \Theta\,N,\qquad \Theta \eqdef \rbr{\frac{1+\lambda}{1-\lambda}}^{\frac{1+\vartheta}2},
    \end{equation*}
    by \eqref{eq:NS} again.
    Altogether $\abs{\cI_P} \le (1+\lambda)G_\star R/A+\Theta N+K_1$, and therefore
    \begin{equation}
        \label{eq:Ksplit}
        K\  = \ \abs{\cJ}+\abs{\cI_S}+\abs{\cI_P}\  \le \ 2K_1+(1+\Theta)N+\frac{(1+\lambda)G_\star R}{A}.
    \end{equation}

    \emph{(iv) Conclusion.} Since $K > 4K_1$ we have $2K_1 < K/2$, so \eqref{eq:Ksplit} gives $(1+\Theta)N+(1+\lambda)G_\star R/A > K/2$ and at least one of the two terms is $\ge K/4$.

    If $(1+\Theta)N \ge K/4$ then, by \eqref{eq:NS} and $c = L_\vartheta/(1+\vartheta)$,
    \begin{equation*}
        A^{2-\vartheta}\  \le \ \frac{8(1+\Theta)(1+\lambda)^{1-\vartheta}}{1+\vartheta}\cdot\frac{L_\vartheta R^2G_\star^{1-\vartheta}}{K}, \qquad\text{i.e.}\qquad A\  \le \ E^{(2)}_\vartheta\rbr{L_\vartheta R^2G_\star^{1-\vartheta}}^{\frac1{2-\vartheta}}K^{-\frac1{2-\vartheta}},
    \end{equation*}
    which is at most $\Lambda_2K^{-1/(2-\vartheta)}$, because $E^{(2)}_\vartheta \le E_\vartheta \le 25$ by \Cref{lem:constants}; 
    with $\lambda = \tfrac23$ one has $\Theta = 5^{(1+\vartheta)/2}$ and $1+\lambda = \tfrac53$.

    If instead $(1+\lambda)G_\star R/A \ge K/4$ then $A \le 4(1+\lambda)G_\star R\,K^{-1} \le E^{(3)}_\vartheta\,G_\star R\,K^{-1/(2-\vartheta)}$, since $1 \ge \tfrac1{2-\vartheta}$ and $K \ge 1$; this too is at most $\Lambda_2K^{-1/(2-\vartheta)}$, since $E^{(3)}_\vartheta = \tfrac{20}3 \le 25$ (\Cref{lem:constants}) and $\Lambda_2 \ge 25\,G_\star R$ by \eqref{eq:lambdas}.

    In every case $A \le \max\{\Lambda_1K^{-(1+\vartheta)/2},\Lambda_2K^{-1/(2-\vartheta)}\}$, which is \eqref{eq:localrate}.
\end{proof}

\end{document}